\UseRawInputEncoding
\documentclass[12pt, reqno]{amsart}
\usepackage[margin=1in]{geometry}
\usepackage{amssymb,latexsym,amsmath,amscd,amsfonts}
\usepackage{latexsym}
\usepackage[mathscr]{eucal}
\usepackage{bm}
\usepackage{mathptmx}
\usepackage{xcolor}
\usepackage{amssymb}
\usepackage{amsthm}
\usepackage{dcolumn}
\usepackage[all]{xy}
\usepackage{enumitem}
\usepackage[utf8]{inputenc}

\def \qed {\hfill \vrule height6pt width 6pt depth 0pt}
\def\textmatrix#1&#2\\#3&#4\\{\bigl({#1 \atop #3}\ {#2 \atop #4}\bigr)}
\def\dispmatrix#1&#2\\#3&#4\\{\left({#1 \atop #3}\ {#2 \atop #4}\right)}
\newcommand{\beg}{\begin{equation}}
	\newcommand{\eeg}{\end{equation}}
\newcommand{\ben}{\begin{eqnarray*}}
	\newcommand{\een}{\end{eqnarray*}}

\newtheorem{thm}{Theorem}[section]
\newtheorem{cor}[thm]{Corollary}
\newtheorem{lem}[thm]{Lemma}

\newtheorem{prop}[thm]{Proposition}
\numberwithin{equation}{section} \theoremstyle{definition}
\newtheorem{defn}[thm]{Definition}
\newtheorem{rem}[thm]{Remark}

\newtheorem{eg}[thm]{Example}

\newcommand{\C}{\mathbb{C}}

\newcommand{\G}{\mathbb{G}_2}
\newcommand{\D}{\mathbb{D}}
\newcommand{\T}{\mathbb{T}}
\newcommand{\N}{\mathbb{N}}

\newcommand{\HS}{\mathcal{H}}

\newcommand{\DC}{\overline{\mathbb{D}}}

\def\textmatrix#1&#2\\#3&#4\\{\bigl({#1 \atop #3}\ {#2 \atop #4}\bigr)}
\def\dispmatrix#1&#2\\#3&#4\\{\left({#1 \atop #3}\ {#2 \atop #4}\right)}

\begin{document}
	
	\title[Constrained dilation and $\Gamma$-contractions]{Constrained dilation and $\Gamma$-contractions}
	
	\author[Pal and Tomar]{SOURAV PAL AND NITIN TOMAR}
	
	\address[Sourav Pal]{Mathematics Department, Indian Institute of Technology Bombay, Powai, Mumbai-400076, India.} \email{sourav@math.iitb.ac.in, souravmaths@gmail.com}	
	
	\address[Nitin Tomar]{Mathematics Department, Indian Institute of Technology Bombay, Powai, Mumbai-400076, India.} \email{tomarnitin414@gmail.com}		
	
	

	\begin{abstract} 
A commuting pair of Hilbert space operators having the closed symmetrized bidisc 
\[
\Gamma=\{(z_1+z_2, z_1z_2) \ : \ |z_1| \leq 1, |z_2| \leq 1\}
\]
as a spectral set is called a \textit{$\Gamma$-contraction}. A \textit{$\Gamma$-unitary} is a commuting pair of normal operators with its Taylor joint spectrum inside the distinguished boundary $b\Gamma$ of $\Gamma$ and a \textit{$\Gamma$-isometry} is the restriction of a $\Gamma$-unitary to a joint invariant subspace of its components. Also, a \textit{pure} $\Gamma$-contraction is a $\Gamma$-contraction $(S, P)$ for which $P$ is a pure contraction, i.e., $P^{*n} \to 0$ strongly as $n \to \infty$. A $\Gamma$-contraction $(S,P)$ is called \textit{$\Gamma$-distinguished} if $(S,P)$ is annihilated by a polynomial $q \in \mathbb C[z_1,z_2]$ whose zero set $Z(q)$ defines a distinguished variety in the symmetrized bidisc $\G$. There is Schaffer-type minimal $\Gamma$-isometric dilation of a $\Gamma$-contraction $(S,P)$ in the literature. In this article, we study when such a minimal $\Gamma$-isometric dilation is $\Gamma$-distinguished provided that $(S,P)$ is a $\Gamma$-distinguished $\Gamma$-contraction. We show that a pure $\Gamma$-isometry $(T,V)$ with defect space $\dim \mathcal D_{V^*}< \infty$, is $\Gamma$-distinguished if and only if the fundamental operator of $(T^*,V^*)$ has numerical radius less than $1$. Further, it is proved that a $\Gamma$-contraction acting on a finite-dimensional Hilbert space dilates to a $\Gamma$-distinguished $\Gamma$-isometry if its fundamental operator has numerical radius less than $1$. We also provide sufficient conditions for a pure $\Gamma$-contraction to be $\Gamma$-distinguished. Wold decomposition splits an isometry into two orthogonal parts of which one is a unitary and the other is a completely non-unitary contraction. In this direction, we find a few decomposition results for the $\Gamma$-distinguished $\Gamma$-unitaries and $\Gamma$-distinguished pure $\Gamma$-isometries.
	\end{abstract}

	\maketitle

	\section{Introduction}\label{sec_intro}

\noindent Throughout the paper, all operators are bounded linear maps acting on complex Hilbert spaces. A contraction is an operator with norm at most $1$. For Hilbert spaces $\HS$ and $\mathcal{L}$, the space of all operators from $\HS$ to $\mathcal{L}$ is denoted by $\mathcal{B}(\HS, \mathcal{L})$ with $\mathcal{B}(\HS)=\mathcal{B}(\HS, \HS)$. We shall use the following notations:  $\C$ denotes the complex plane, $\D$ is the open unit disc $\{z\in \C : |z|<1\}$, $\T$ is the unit circle $\{z\in \C : |z|=1\}$ and $\mathbb{E}=\C \setminus \DC$, i.e., the complement of $\DC$. Given an operator $T$, the numerical and spectral radii of $T$ are denoted by $\omega(T)$ and $r(T)$ respectively.

\smallskip 

The symmetrized bidisc $\mathbb G_2$ is a non-convex but polynomially convex domain in $\C^2$ defined by
\[
\mathbb{G}_2:=\{(z_1+z_2, z_1z_2) \ : |z_1|<1, |z_2|<1\}. 
\]	
Evidently, a $2 \times 2$ matrix $A$ has spectral radius less than $1$ if and only if its eigenvalues are in $\D$, which is equivalent to saying that $(\text{tr}(A), \text{det}(A))$ belongs to $\G$. Furthermore, we can simply write $\G=\pi(\D^2)$ and its closure as 
\begin{align*}
	\Gamma= \overline{\mathbb G}_2=\pi(\DC^2)=\{(z_1+z_2, z_1z_2) \ : \ |z_1|\leq 1, |z_2| \leq 1 \},
\end{align*}
where $\pi: \C^2 \to \C^2$ is the symmetrization map defined as $\pi(z_1, z_2)=(z_1+z_2, z_1z_2)$. The symmetrized bidisc $\G$ originates from the $2 \times 2$ spectral Nevanlinna-Pick problem, which can be reformulated into a similar interpolation problem between $\D$ and $\G$ as was shown in \cite{Nikolov}. Furthermore, $\G$ is the first example of a non-convex domain in which the Caratheodory and Kobayashi distances coincide, e.g., see \cite{AglerIV}. Such interesting properties have made the symmetrized bidisc $\G$ a central object of study from complex geometric, function theoretic and operator theoretic points of view, e.g., see \cite{AglerI15, AglerII16, AglerIV, Bhatt-Pal, Pal8, PalShalit1} and references therein.

\smallskip 

The primary objects of study in this article are distinguished varieties in 
$\G$, the polynomials defining distinguished varieties in $\G$ and the operators associated with $\G$. A commuting pair of operators $(S,P)$ acting on a Hilbert space $\HS$ is said to be a $\Gamma$-\textit{contraction} if $\Gamma$ is a spectral set for $(S,P)$, i.e., the Taylor joint spectrum $\sigma_T(S,P) \subseteq \Gamma$ and von Neumann's inequality
\[
\|f(S,P)\| \leq \sup\{|f(s,p)| : (s, p) \in \Gamma\}= \|f\|_{\infty, \Gamma}
\]
holds for all rational functions $f=p \slash q$ with $p,q \in \C[z_1,z_2]$ and $q$ having no zeros in $\Gamma$. The class of such rational functions is denoted by Rat($\Gamma$). Just as contractions have special subclasses like unitaries, isometries and pure contractions, analogous subclasses can be defined for $\Gamma$-contractions. Let $(S, P)$ be a commuting operator pair on a Hilbert space $\HS$. The pair $(S, P)$ is called a $\Gamma$-\textit{unitary} if $S, P$ are normal and $\sigma_T(S,P)$ is contained in the distinguished boundary $b\Gamma$ of $\Gamma$ given by
\[
b\Gamma=\pi(\T^2)=\{(z_1+z_2, z_1z_2): z_1, z_2 \in \T\}.
\]
We call $(S, P)$ a $\Gamma$-\textit{isometry} if there is a Hilbert space $\mathcal{K} \supseteq \HS$ and a $\Gamma$-unitary $(T, U)$ acting on $\mathcal{K}$ such that $\HS$ is a common invariant subspace for $T, U$ and $(S, P)=(T|_{\HS}, U|_{\HS})$. Also, a $\Gamma$-contraction $(S, P)$ is said to be a \textit{pure} if $P$ is a pure contraction, i.e., ${P^*}^n \rightarrow 0$ strongly as $n \rightarrow \infty$.

\begin{defn}
Given a bounded domain $\Omega$ in $\C^n$, a nonempty set $V \subseteq \Omega$ is said to be a \textit{distinguished variety} in $\Omega$ if there is an algebraic variety $W \subset \C^n$ such that $V=W \cap \Omega $ and $W \cap \partial \overline{\Omega}= W \cap b\overline{\Omega}$, where $b\overline{\Omega}$ is the distinguished boundary of $\overline{\Omega}$.
\end{defn}

Over the past two decades, distinguished varieties in the domains such as the bidisc, symmetrized bidisc, tetrablock and more generally, in the polydisc and symmetrized polydisc have attracted significant attention, e.g., see \cite{AglerKneseMcCarthy2, AglerMcCarthy, Bhatt-Sau-Kumar, Das, DasII, Dritschel, Knese9, PalS12, PalS, PalS11, PalShalit1}. One of the most pioneering works in operator theory is And\^{o}'s inequality \cite{Ando}, which states that if $(T_1, T_2)$ is a commuting pair of contractions, then for every polynomial $p \in \C[z_1, z_2]$,
 \[
 \|p(T_1, T_2)\|\leq \sup\{|p(z_1, z_2)| : (z_1, z_2) \in \DC^2 \}.
 \] 
 The main motivation behind studying distinguished varieties comes from a refinement of And\^{o}'s inequality due to Agler and McCarthy \cite{AglerMcCarthy}, where they proved the inequality for a pair of commuting contractive matrices $(T_1, T_2)$ under additional hypothesis, with the supremum taken over a distinguished variety in $\D^2$. An analogous result for the symmetrized bidisc was proved by the authors of \cite{PalShalit1}. Also, an explicit description of all distinguished varieties in the bidisc and the symmetrized bidisc were given in \cite{AglerMcCarthy} and \cite{PalShalit1}, respectively. To go parallel with the bidisc, we say that a polynomial $p \in \C[z_1, z_2]$ is \textit{$\Gamma$-distinguished} if its zero set $Z(p)$ defines a distinguished variety with respect to $\G$, i.e., $Z(p) \cap \G \ne \emptyset$ and $Z(p) \cap \partial \Gamma =Z(p) \cap b\Gamma$. Also, a $\Gamma$-contraction $(S, P)$ is called \textit{$\Gamma$-distinguished} if it is annihilated by a $\Gamma$-distinguished polynomial.		
 
 \smallskip 

Agler and Young \cite{AglerI15} proved that every $\Gamma$-contraction $(S, P)$ on a space $\HS$ admits a $\Gamma$-isometric dilation, i.e., there exists a $\Gamma$-isometry $(T, V)$ on a space $\mathcal{K} \supseteq \HS$ such that $f(S, P)=P_\HS f(T, V)|_\HS$ for all $f \in \text{Rat}(\Gamma)$, where $P_\HS$ is the orthogonal projection of $\mathcal{K}$ onto $\HS$. Such a dilation is referred to as a $\Gamma$-unitary dilation if $(T, V)$ is a $\Gamma$-unitary. An explicit construction of a $\Gamma$-isometric dilation of $(S, P)$ on the minimal isometric dilation space of $P$ was established by the authors of \cite{Pal8} for which the notion of the fundamental operator of a $\Gamma$-contraction was introduced. The success of $\Gamma$-isometric dilation for $\Gamma$-contractions naturally raises the question if every $\Gamma$-distinguished $\Gamma$-contraction dilates to a $\Gamma$-distinguished $\Gamma$-isometry. While the general problem is still open, the authors in \cite{Pal_Tomar} provided several characterizations of $\Gamma$-distinguished $\Gamma$-contractions that admit such dilations. For a contraction $P$, we denote by $D_P=(I-P^*P)^{1\slash 2}$ and $\mathcal{D}_P=\overline{\text{Ran}} \ D_P$ the defect operator and defect space respectively of $P$, respectively.

\vspace{0.15cm}

In this article, we continue this study and show that $\Gamma$-distinguished $\Gamma$-contractions admit $\Gamma$-distinguished $\Gamma$-isometric dilations in particular cases. A first step in this direction is to determine which $\Gamma$-isometries are $\Gamma$-distinguished.  Although the authors in \cite{Pal_Tomar} provides a characterization of such $\Gamma$-isometries, we further refine it under certain assumption. It is proved in Section \ref{sec05} that a pure $\Gamma$-isometry $(S, P)$ with finite-dimensional $\mathcal{D}_{P^*}$ is $\Gamma$-distinguished if and only if $(F_*, 0)$ is $\Gamma$-distinguished, which holds if and only if $r(F_*)<1$, where $F_*$ is the fundamental operator of $(S^*, P^*)$. As an application of this result and the fact that every pure $\Gamma$-contraction dilates to a pure $\Gamma$-isometry (see \cite{Pal8}), a sufficient condition for pure $\Gamma$-contractions $(S, P)$ with $\dim \mathcal{D}_{P^*}$ being finite is established in Theorem \ref{5.9}. Finally, we prove in Theorem \ref{thm_FD} that $\Gamma$-contractions on finite-dimensional Hilbert spaces dilate to $\Gamma$-distinguished $\Gamma$-isometries if the fundamental operator has numerical radius strictly less than $1$, and we provide two independent proofs of this result.

\vspace{0.15cm}

Since it is still unknown whether $\Gamma$-distinguished $\Gamma$-contractions always admit $\Gamma$-distinguished $\Gamma$-isometric dilations, it is natural to consider the $\Gamma$-isometric dilation constructed in \cite{Pal8}. In Section \ref{minimal}, we recall the $\Gamma$-isometric dilation $(T_A, V_0)$ of a $\Gamma$-contraction $(S, P)$ constructed therein, and address the following key questions:

\vspace{0.1cm}

(1) Is $(T_A, V_0)$ necessarily $\Gamma$-distinguished, at least when $(S,P)$ itself is $\Gamma$-distinguished?
	
	\vspace{0.15cm} 
	
	(2) If both $(S,P)$ and $(T_A, V_0)$ are $\Gamma$-distinguished, then does every $\Gamma$-distinguished polynomial annihilating $(S,P)$ also annihilate $(T_A, V_0)$?

\vspace{0.1cm}

Examples \ref{7.4} and \ref{Shift_counter_example} show that the above conclusions do not hold in general. Example \ref{Shift_counter_example} also shows that a $\Gamma$-contraction can dilate to two different $\Gamma$-isometries of which one is $\Gamma$-distinguished and the other is not. This, in turn, leads to the problem of characterizing when $(T_A, V_0)$ is $\Gamma$-distinguished. In this direction, we prove in Theorem \ref{thm:Char-f} that if $(S,P)$ is a $\Gamma$-distinguished $\Gamma$-contraction with the fundamental operator $A \in \mathcal{B}(\mathcal{D}_P)$ and finite-dimensional $\mathcal{D}_P$, then $(T_A, V_0)$ is $\Gamma$-distinguished if and only if $(A,0)$ is $\Gamma$-distinguished, equivalently when $r(A) < 1$. Moreover, we show that this equivalence remains true if $A$ is hyponormal, even when $\dim \mathcal{D}_{P}$ is not finite (see Remark \ref{rem6.7} and Theorem \ref{7.12}). In general, we prove that if $(S,P)$ and $(A,0)$ are $\Gamma$-distinguished, then $(T_A, V_0)$ can be realized as the limit of a sequence of $\Gamma$-distinguished $\Gamma$-contractions in the strong operator topology. At every stage, the fundamental operator of $(S, P)$ plays a crucial role.

\vspace{0.15cm} 

A Wold type decomposition for $\Gamma$-isometries from \cite{AglerII16} states that any $\Gamma$-isometry splits into a $\Gamma$-unitary and a pure $\Gamma$-isometry. Hence, if a $\Gamma$-isometry is $\Gamma$-distinguished, then each component is also $\Gamma$-distinguished, and the same polynomial annihilates both parts. Thus, every $\Gamma$-distinguished $\Gamma$-isometry splits into a $\Gamma$-distinguished unitary and a $\Gamma$-distinguished pure $\Gamma$-isometry. In Section \ref{sec_decomp}, we provide a further decomposition for a subclass of such operator pairs, namely those which are annihilated by polynomials whose zero sets lie in $\G \cup b\Gamma \cup \pi(\mathbb{E}^2)$. This decomposition is motivated by Theorem 2.1 in \cite{AglerKneseMcCarthy2}, where a similar decomposition was established for pure isometries based on the irreducible factors of an annihilating polynomial $q$ with $Z(q) \subseteq \D^2 \cup \T^2 \cup \mathbb{E}^2$. 

 \vspace{0.1cm}
	
	\section{The $\Gamma$-distinguished $\Gamma$-contractions on finite-dimensional Hilbert spaces}\label{sec05}

\vspace{0.1cm}	

\noindent In this section, we primarily study the $\Gamma$-distinguished $\Gamma$-isometric dilations of $\Gamma$-contractions $(S, P)$ acting on finite-dimensional Hilbert spaces. A natural first step is to identify which $\Gamma$-isometries are $\Gamma$-distinguished. While a characterization is established by the authors in \cite{Pal_Tomar}, it can be further refined if the defect space of $P^*$ is finite-dimensional for a pure $\Gamma$-isometry $(S, P)$. To begin with, we show that being $\Gamma$-distinguished is invariant under unitary equivalence and the proof is straightforward.

	\begin{lem}\label{lem5.5}
		If $(S_1, P_1)$ and $(S_2, P_2)$ are two unitarily equivalent $\Gamma$-contractions on the Hilbert spaces $\mathcal{H}_1$ and $\mathcal{H}_2$ respectively, i.e., there is a unitary operator $U: \mathcal{H}_1 \to \mathcal{H}_2$ such that $S_1=U^*S_2U$ and $ P_1=U^*P_2U$. Then $(S_1, P_1)$ is $\Gamma$-distinguished if and only if $(S_2, P_2)$ is $\Gamma$-distinguished.
	\end{lem}

	\begin{proof}
		The conclusion follows from the fact that if $p \in \mathbb{C}[z_1, z_2]$ is a polynomial annihilating $(S_2,P_2)$, then $p(S_1, P_1)=U^*p(S_2, P_2)U$.
	\end{proof}

The rich operator theory of $\G$ is based on the following fundamental results from \cite{AglerI15, Pal8}. To understand these results better, we recall (see Section I.3 of \cite{NagyFoias6}) that $D_P=(I-P^*P)^{1\slash 2}$ and $\mathcal{D}_P=\overline{\text{Ran}} \ D_P$ are the defect operator and defect space respectively of a contraction $P$.

	\begin{thm}[\cite{AglerI15}, Theorem 1.2 \& \cite{Pal8}, Theorem 4.4]\label{thm_201}
	Let $(S,P)$ be a pair of commuting operators on a Hilbert space $\mathcal{H}$. Then the following are equivalent:
	\begin{enumerate}
		\item $(S,P)$ is a $\Gamma$-contraction;
		\item $(S,P)$ admits a $\Gamma$-isometry dilation;
		\item $\|S\| \leq 2, \|P\| \leq 1$ and the operator equation $S-S^*P=D_PXD_P$
has a unique solution $A$ in $\mathcal{B}(\mathcal{D}_P)$ with $\omega(A) \leq 1$.
	\end{enumerate}
The unique operator $A$ is called the \textit{fundamental operator} of the $\Gamma$-contraction $(S,P)$. 	
\end{thm}

For a $\Gamma$-contraction, Theorem \ref{thm_201} ensures the existence of a $\Gamma$-isometric dilation. Moreover, Theorem 4.6 of \cite{Pal8} shows that a pure $\Gamma$-contraction dilates to a pure $\Gamma$-isometry. Thus, pure $\Gamma$-isometries naturally arises as dilation models in this framework. The next theorem provides an explicit model of pure $\Gamma$-isometries using the fundamental operator of the adjoint pair.
  
\begin{thm}[\cite{PalShalit1}, Theorem 2.16]  \label{thm:modelpure}
	Let $(S,P)$ be a pair of commuting operators on a Hilbert space $\mathcal{H}$. If $(S,P)$ is a pure $\Gamma$-isometry, then there is a unitary operator $U:\mathcal{H} \to H^2(\mathcal{D}_{P^*})$ such that 
	\[
	S=U^*T_\phi U \quad \mbox{and} \quad P=U^*T_zU,
	\]
	where $\phi(z)=F_*^*+F_*z$ with	$F_* \in B(\mathcal{D}_{P^*})$ being the fundamental operator of $(S^*, P^*)$, and $T_\phi$ and $T_z$ are Toeplitz operators on the vector-valued Hardy Hilbert space $H^2(\mathcal{D}_{P^*})$ with symbols $\phi$ and $z$ respectively.
\end{thm}

A major part of operator theory on the distinguished varieties in the symmetrized bidisc is based on the $\Gamma$-contractions $(S,P)$ for which $\mathcal D_P$ or $\mathcal D_{P^*}$ is finite-dimensional, e.g., see \cite{PalShalit1}. So, in this section, we explore relations of these classes of $\Gamma$-contractions and their associated fundamental operators with $\Gamma$-distinguished polynomials. More precisely, we provide below a set of necessary and sufficient conditions for such pure $\Gamma$-isometries to be $\Gamma$-distinguished.

	\begin{thm} \label{r(A)<1_(S,P)_distt.}
		Let $(S,P)$ be a pure $\Gamma$-isometry on a Hilbert space $\mathcal{H}$ such that $\dim \mathcal{D}_{P^*} < \infty$ and let $F_* \in B(\mathcal{D}_{P^*})$ be the fundamental operator of $(S^*, P^*)$. Then the following are equivalent:
		\begin{enumerate}
			\item $(F_*,0)$ is $\Gamma$-distinguished;
			\item $(S,P)$ is $\Gamma$-distinguished;
			\item $r(F_*)<1$.
		\end{enumerate}	
	\end{thm} 

	\begin{proof}
 The model theory of pure $\Gamma$-isometry (see Theorem \ref{thm:modelpure}) tells us that $(S,P)$ is unitarily equivalent to the pure $\Gamma$-isometry $(T_{F_*^*+F_*z}, T_z)$ on $H^2(\mathcal{D}_{P^*})$. Since $\dim \mathcal{D}_{P^*} < \infty$, we have that $p(z_1, z_2)=det(F_*^*+F_*z_2-z_1I)$ defines a polynomial. Moreover, $p(T_{F_*^*+F_*z}, T_z)=0$. It follows from Lemma \ref{lem5.5} that $p(S, P)=0$. We prove : $(1) \implies (3) \implies (2) \implies (1)$. 
		
		\medskip
		
	\noindent  $(1)\implies (3)$.  Assume that $f(F_*,0)=0$ for some $\Gamma$-distinguished polynomial $f$. Since $F_*$ is the fundamental operator of $(S^*, P^*),$ we have $r(F_*) \leq \omega(F_*) \leq 1$. Let if possible $r(F_*)=1$. Since $F_*$ is a matrix, there exists some $\lambda \in  \sigma(F_*) \cap \mathbb{T}$. By spectral mapping theorem, we have 
		\begin{equation*}
			\begin{split}
				\{0\}&=\sigma(f(F_*,0))
				=f(\sigma_T(F_*,0))=f(\sigma(F_*)\times \{0\}).\\
			\end{split}
		\end{equation*}
		Since $\lambda \in \sigma(F_*)$, we have that $f(\lambda,0)=0$. Now, $f$ being $\Gamma$-distinguished implies that 
		\[
		\pi(\lambda,0)=(\lambda,0) \in Z(f) \cap \partial \mathbb{G}_2 \subseteq b\Gamma \quad \text{and so,} \quad (\lambda,0) \in b\Gamma,
		\]
		which is a contradiction. Hence, $r(F_*) <1$.
			
\medskip 
				
		\noindent $(3) \implies (2)$. Let $r(F_*) <1$. For a matrix $F$ with numerical radius at most $1$, Theorem 2.5 in \cite{Das} states that
		the set $W=\{(z_1, z_2) \in \C^2: \det(F^*+z_2F-z_1)=0\}$ is a distinguished variety in $\G$ if and only if $\sigma(F) \cap \T =\emptyset$. Since $\sigma(F_*) \subseteq \mathbb{D}$ and $\omega(F_*) \leq 1$, the polynomial $p(z_1, z_2)=det(F_*^*+F_*z_2-z_1I)$ is  $\Gamma$-distinguished. So, $(T_{F_*^*+F_*z}, T_z)$ is $\Gamma$-distinguished as $p(T_{F_*^*+F_*z}, T_z)=0$. 
	
\medskip 
		
		\noindent $(2) \implies (1)$. Let $(S,P)$ be $\Gamma$-distinguished. By Lemma \ref{lem5.5}, $(T_{F_*^*+F_*z}, T_z)$ is also $\Gamma$-distinguished. Let $q$ be a $\Gamma$-distinguished polynomial such that $q(T_{F_*^*+F_*z}, T_z)=0$. It is evident that $q(F_*,0)=0$ and so, $(F_*, 0)$ is $\Gamma$ -distinguished. The proof is now complete.
		\end{proof}

As an application of Theorem \ref{r(A)<1_(S,P)_distt.}, we obtain a sufficient condition for a pure $\Gamma$-contraction to be $\Gamma$-distinguished under some additional assumption.

	\begin{thm}\label{5.9}
		Let $(S,P)$ be a pure $\Gamma$-contraction on a Hilbert space $\mathcal{H}$ such that $\dim \mathcal{D}_{P^*} < \infty$. Let $F_* \in B(\mathcal{D}_{P^*})$ be the fundamental operator of $(S^*, P^*)$. Then $(S,P)$ is $\Gamma$-distinguished if $(F_*,0)$ is $\Gamma$-distinguished or $r(F_*)<1$.
	\end{thm} 
	\begin{proof}
		By Theorem 4.6 in \cite{Pal8}, the $\Gamma$-isometry $(T,V)=(T_{F_*^*+F_*z}, T_z)$ on $H^2(\mathcal{D}_{P^*})$ dilates $(S,P)$, which means that
		$
		f(S,P)=P_\mathcal{H}f(T_{F_*^*+F_*z}, T_z)|_\mathcal{H}
		$
		for every polynomial $f \in \mathbb{C}[z_1, z_2]$. The pair $(T,V)$ is a pure $\Gamma$-isometry with $F_*$ as the fundamental operator of $(S^*, P^*)$. Since $\dim\mathcal{D}_{P^*} < \infty,$ Theorem \ref{r(A)<1_(S,P)_distt.} yields that $(T_{F_*^*+F_*z}, T_z)$ is $\Gamma$-distinguished if and only if $(F_*, 0)$ is $\Gamma$-distinguished or $r(F_*)<1$. In either case, there is a $\Gamma$-distinguished polynomial annihilating $(T,V)$. Clearly, the same polynomial annihilates $(S,P)$, which completes the proof.
	\end{proof}
	
	It will be shown in Example \ref{Shift_counter_example} that the converse of the above theorem fails in general. We next present another sufficient condition for $\Gamma$-contractions acting on finite-dimensional spaces to be $\Gamma$-distinguished. To do so, we recall the following result from \cite{PalShalit1}, which provides an explicit description of a distinguished variety in the symmetrized bidisc. 
		 
	 \begin{thm} [\cite{PalShalit1}, Theorem 3.5]\label{PalShalit}
	 	For a square matrix $A$ with $\omega(A) <1$, the set $W$ given by 
	 	\[
	 	W=\{(z_1, z_2)\in \mathbb{G}_2 \ :\det(A+ z_2A^*-z_1I)=0\}.
	 	\]
	 	is a distinguished variety in $\G$. Conversely, every distinguished variety in $\mathbb{G}_2$ has the form $\{(z_1, z_2)\in \mathbb{G}_2 \ :\ det(A+ z_2A^*-z_1I)=0\}$ for some matrix $A$ with $\omega(A) \leq 1$.
	 \end{thm}

The above theorem leads to the following result.

\begin{prop}\label{prop512}
	Let $(S,P)$ be a $\Gamma$-contraction on a finite-dimensional Hilbert space $\mathcal{H}$ and let $A$ be the fundamental operator of $(S,P)$ with  $\omega(A)<1$. Then $(S,P)$ is $\Gamma$-distinguished.
\end{prop}
\begin{proof}
	It follows from Theorem \ref{PalShalit} that $p(z_1,
	z_2)=\det(A^*+z_2A-z_1I)$ is a $\Gamma$-distinguished polynomial as
	$\omega(A) <1$. By Theorem 4.7 in \cite{PalShalit1}, we have 
	\[
	\|p(S, P)\| \leq \sup\{|p(z_1, z_2)|: \det(A^*+z_2A-z_1I) =0 \}=0.
	\]
	Therefore, the $\Gamma$-distinguished polynomial $p$ annihilates $(S,P)$.
\end{proof}

The hypothesis that $\omega(A)<1$ cannot be dropped in the above proposition and the reason is explained in the following example that appears as Example 4.2 in \cite{Pal_Tomar}.

\begin{eg}\label{3.14}
	Consider the matrix
	\[
	A=\begin{pmatrix}
		0 & 2 & 0\\
		0 & 0 & 0\\
		0 & 0 & 1\\
	\end{pmatrix}.
	\]
	It is not difficult to see that $\omega(A) =1$ and by Theorem \ref{thm_201}, $(A, 0)$ is a $\Gamma$-contraction on $\HS=\C^3$. Let if possible, $q(A, 0)=0$ for some $\Gamma$-distinguished polynomial $q(z_1, z_2)$. By spectral mapping theorem, $\sigma(q(A,0))=\{ q(\lambda , 0)\,:\, \lambda \in \sigma(A)\}=\{0\}$ and so, $q(1, 0) =0$. Clearly, $(1,0) \in \partial \Gamma = \Gamma \setminus \mathbb G_2$ by being the symmetrization of the points $0,1$. As $q$ is $\Gamma$-distinguished, we have that $(1,0) \in Z(q) \cap \partial \Gamma \subseteq b\Gamma$, which is a contradiction as $(1,0)  \notin b\Gamma$. Thus, $(A, 0)$ is not $\Gamma$-distinguished.
	
	\medskip 
	
	Also, the pair $(A, 0)$ cannot admit a $\Gamma$-distinguished $\Gamma$-isometric dilation. Let if possible, $(T, V)$ be a $\Gamma$-isometric dilation of $(A, 0)$ and let $p(T, V)=0$ for a $\Gamma$-distinguished polynomial $p$. Then $p(A, 0)=P_\HS p(T_, V)|_\HS=0$ and so, $(A, 0)$ is $\Gamma$-distinguished, which is a contradiction.
	\qed
\end{eg}

The success of dilation of $\Gamma$-contractions to $\Gamma$-isometries naturally leads to the question if every $\Gamma$-distinguished $\Gamma$-contraction dilates to a $\Gamma$-distinguished $\Gamma$-isometry. While the general problem remains open, the authors of this article provided the characterizations of $\Gamma$-distinguished $\Gamma$-contractions that admit such a dilation. A part of this theorem from \cite{Pal_Tomar} is written below. 

\begin{thm}[\cite{Pal_Tomar}, Theorem 1.7 \&  Proposition 7.1]\label{thm814}	
	Let $(S, P)$ be a $\Gamma$-distinguished $\Gamma$-contraction acting on a Hilbert space $\HS$. Then the following are equivalent:
	\begin{enumerate}
		
		\item $(S, P)$ dilates to a $\Gamma$-distinguished $\Gamma$-isometry;
		\item $(S, P)$ dilates to a $\Gamma$-distinguished $\Gamma$-unitary;
		\item there exists a $\Gamma$-distinguished polynomial $p \in \C[z_1, z_2]$ such that $Z(p) \cap \Gamma$ is a complete spectral set for $(S,P)$, i.e.,
		\[
		\|[f_{ij}(S, P)]_{m \times m}\| \leq \sup\{\|[f_{ij}(z_1, z_2)]_{m\times m}\|  : (z_1, z_2) \in Z(p) \cap \Gamma\}
		\]
		for all matricial functions $[f_{ij}]_{m \times m}$ with each $f_{ij}$ in the rational algebra $\text{Rat}(Z(p) \cap \Gamma)$. 
	\end{enumerate}		
\end{thm}

Capitalizing Theorem \ref{thm814}, our next result shows that every $\Gamma$-contraction on a finite-dimensional Hilbert space admits dilation to a $\Gamma$-distinguished $\Gamma$-unitary if its fundamental operator has numerical radius strictly less than one. In fact, we present two different proofs of this result.

\begin{thm}\label{thm_FD}
	Let $(S,P)$ be a $\Gamma$-contraction on a finite-dimensional Hilbert space $\mathcal{H}$ and let $A$ be the fundamental operator of $(S,P)$ with  $\omega(A)<1$. Then $(S,P)$ dilates to a $\Gamma$-distinguished $\Gamma$-unitary.
\end{thm}	

\begin{proof} We have by Theorem \ref{PalShalit} that $p(z_1, z_2)=\det(A^*+z_2A-z_1I)$ is a $\Gamma$-distinguished polynomial as $\omega(A)<1$. It follows from Theorem 4.7 in \cite{PalShalit1} that 
	\[
	\|f(S, P)\| \leq \max\{\|f(z_1, z_2)\|: (z_1, z_2) \in Z(p) \cap \Gamma \}
	\]
	for every matrix valued polynomial $f$. Since $Z(p) \cap \Gamma$ is polynomially convex (see Lemma 5.5 in \cite{Pal_Tomar}), we have that $Z(p) \cap \Gamma$ is a complete spectral set for $(S, P)$. It now follows from Theorem \ref{thm814} that $(S, P)$ dilates to a $\Gamma$-distinguished $\Gamma$-unitary.	
\end{proof} 

For an alternative proof, we proceed as in the proof of Theorem 4.8 in \cite{DasII}. 

\begin{proof} 	
	It follows from Proposition \ref{prop512} that $(S, P)$ is $\Gamma$-distinguished. Thus, the joint spectrum $\sigma_T(S, P)$ does not intersect $\pi(\T \times \D)$ or $\pi(\D \times \T)$. By simultaneous upper triangularisation, we have (up to a unitarily equivalence)
	\[
	S=\begin{bmatrix}
		S_1 & S_3 \\
		0 & S_2
	\end{bmatrix} \quad \text{and} \quad P=\begin{bmatrix}
		P_1 & P_3 \\
		0 & P_2
	\end{bmatrix},
	\]
	with respect to a decomposition $\HS=\HS_1 \oplus \HS_2$ such that $S_1, S_2, P_1, P_2$ are square matrices with $\sigma(P_1) \subseteq \D$ and $\sigma(P_2) \subseteq \T$. Since $P_2$ is upper triangular with $\|P_2\|\le1$ and spectrum on $\T$, it must be unitary. Applying the same reasoning and using $\|P\|\le1$, we have that $P_3=0$. Furthermore, Lemma 2.1 in \cite{DasII} ensures that $P_1$ and $P_1^*$ are pure contractions as $\sigma(P_1) \subseteq \D$. By the commutativity of $S$ and $P$, we have $S_1P_3+S_3P_2=P_1S_3+P_3S_2$ and so, $S_3P_2=P_1S_3$. Since $P_1$ is a pure contraction and $P_2$ is unitary, it follows that $S_3=0$. Putting everything together, we have
	\[
	S=\begin{bmatrix}
		S_1 & 0 \\
		0 & S_2
	\end{bmatrix} \quad \text{and} \quad P=\begin{bmatrix}
		P_1 & 0 \\
		0 & P_2
	\end{bmatrix},
	\]
	where $P_1, P_1^*$ are pure contractions and $P_2$ is a unitary. Consequently, $(S_1^*, P_1^*)$ is a pure $\Gamma$-contraction. Since $P_2$ is a unitary, we have $D_{P_2}=0$, and by Theorem \ref{thm_201}, $S_2 = S_2^* P_2$. As $P_2$ is a diagonal unitary and $S_2$ is upper triangular, the relation $S_2 = S_2^* P_2$ implies that $S_2$ is also diagonal. Therefore, $(S_2, P_2)$ is a normal $\Gamma$-contraction with $\sigma_T(S_2, P_2) \subseteq b\Gamma$, and hence $(S_2, P_2)$ is a $\Gamma$-unitary.
		 Evidently, $(S_2, P_2)$ is also $\Gamma$-distinguished since $(S, P)$ is $\Gamma$-distinguished. Since the fundamental operator of $(S, P)$ has numerical radius strictly less than $1$, the same holds for the fundamental operator of $(S_1, P_1)$. It follows from Theorem 7.17 in \cite{Pal_Tomar} that $(S_1, P_1)$ admits a dilation to a $\Gamma$-distinguished $\Gamma$-unitary $(T, U)$ on a space $\mathcal{L}$ containing $\HS_1$. Consider the commuting operator pair given by
	\[
	V_1=\begin{bmatrix}
		T & 0 \\
		0 & S_2
	\end{bmatrix} \quad \text{and} \quad V_2=\begin{bmatrix}
		U & 0 \\
		0 & P_2
	\end{bmatrix}
	\]
	on the Hilbert space $\mathcal{K}=\mathcal{L}\oplus \mathcal{H}_2$. Clearly, $(V_1, V_2)$ is a $\Gamma$-unitary that dilates $(S, P)$. Moreover, if $f$ and $g$ are $\Gamma$-distinguished polynomials, annihilating $(T, U)$ and $(S_2, P_2)$ respectively, then the $\Gamma$-distinguished polynomial $fg$ annihilates $(V_1, V_2)$. The proof is now complete.
\end{proof}	

We conclude this section by presenting a few concrete examples of $\Gamma$-distinguished polynomials.

\begin{eg}
	For $|c|<2$, define $A=\begin{pmatrix}
		0 & 0 \\
		c & 0
	\end{pmatrix}$. Then $\det(A+z_2A^*-z_1I)=z_1^2-|c|^2z_2$. By Theorem \ref{PalShalit},  $p(z_1, z_2)=z_1^2-|c|^2z_2$ is a $\Gamma$-distinguished polynomial since $\omega(A) =|c|\slash 2<1$. 
\end{eg}

\begin{eg} \label{3.4}
	This example also appears in \cite{Pal_Tomar}. For the polynomial $q(z_1, z_2)=4z_2- z_1^2$, it is clear that $(0, 0) \in Z(q) \cap \G$. Let $(z_1, z_2) \in Z(q) \cap \partial \G$. Then $(z_1, z_2)=(\alpha+\beta, \alpha \beta)$ for some $(\alpha, \beta) \in \partial \D^2$. Since $z_1^2-4z_2=(\alpha+\beta)^2-4\alpha\beta=0$, we have that $\alpha=\beta$. So, $(\alpha, \beta) \in \T^2$ and  $(z_1, z_2) \in b\Gamma$. Consequently, $Z(q) \cap \partial \G \subseteq b\Gamma$ and $q(z_1, z_2)$ is $\Gamma$-distinguished. \qed 
\end{eg}

\section{Minimal dilation of a $\Gamma$-distinguished $\Gamma$-contraction}\label{minimal}	
		
	\vspace{0.1cm}

\noindent Recall that a $\Gamma$-contraction $(S, P)$ on a Hilbert space $\HS$ is said to admit a $\Gamma$-isometric dilation (or $\Gamma$-unitary dilation) if there exists a $\Gamma$-isometry (or $\Gamma$-unitary) $(T, V)$ acting on a Hilbert space $\mathcal{K}$ containing $\HS$ such that 
\begin{equation}\label{eqn_3001}
f(S, P)=P_\HS f(T, V)|_{\HS}
\end{equation}
for all $f \in \text{Rat}(\Gamma)$. Since $\Gamma$ is polynomially convex, \eqref{eqn_3001} is equivalent to saying that $p(S, P)=P_\HS p(T, V)|_{\HS}$ for all $p \in \C[z_1, z_2]$. Furthermore, a $\Gamma$-isometric dilation $(T, V)$ is called minimal if $\mathcal{K}=\overline{\text{span}}\{p(T, V)h: p \in \C[z_1, z_2], \ h \in \HS \}$. It was shown by Agler and Young in \cite{AglerI15} that every $\Gamma$-contraction admits a $\Gamma$-isometric dilation. Also, a minimal $\Gamma$-isometric dilation of a $\Gamma$-contraction was explicitly constructed in \cite{Pal8} with the help of the fundamental operator (see Part (3) of Theorem \ref{thm_201}). Also, that dilation was minimal and acted on the minimal isometric dilation space of $P$. In this section, we study in detail when this particular minimal $\Gamma$-isometric dilation of a $\Gamma$-contraction $(S,P)$ is annihilated by a $\Gamma$-distinguished polynomial, especially when dimension of $\mathcal D_P$ or $\mathcal D_{P^*}$ is finite. We mention the dilation theorem from \cite{Pal8}, which also described a new way of characterizing $\Gamma$-contractions.

	\begin{thm}[\cite{Pal8}, Theorem 4.3]\label{Tirtha-Pal} 
		Let $(S,P)$ be a $\Gamma$-contraction on a Hilbert space $\mathcal{H}$. Let $A$ be the unique fundamental operator of the fundamental equation $S-S^*P=D_PAD_P$. Consider the operators $T_A, V_0$ defined on $\mathcal{H}\bigoplus \ell^2(\mathcal{D}_P)$ by 
		\begin{equation*}
			\begin{split}
				& T_A(x_0, x_1, x_2, \dots)=(Sx_0, \; A^*D_Px_0+Ax_1, \; A^*x_1+Ax_2, \; A^*x_2+Ax_3, \dots),\\
				&		V_0(x_0, x_1, x_2, \dots)=(Px_0, \; D_Px_0, \; x_1, \; x_2, \dots).  \\
			\end{split}		
		\end{equation*}
		Then up to unitary $(T_A, V_0)$ is the unique $\Gamma$-isometric dilation of $(S,P)$ on $\mathcal H \oplus \ell^2(\mathcal D_P)$. 
	\end{thm}
	
For a $\Gamma$-contraction $(S, P)$ on a space $\HS$, the $\Gamma$-isometric dilation $(T_A, V_0)$ on $\mathcal H \oplus \ell^2(\mathcal D_P)$ constructed in Theorem \ref{Tirtha-Pal} is minimal, i.e., 
	$
	\mathcal H \oplus \ell^2(\mathcal D_P)=\overline{\text{span}}\{p(T_A, V_0)h: p \in \C[z_1, z_2], \ h \in \HS \}.
	$
Note that the operators $T_A, V_0$ have the following matrix form with respect to the decomposition $\mathcal H \oplus \mathcal D_P \oplus \mathcal D_P \oplus \dotsc$ of the space $\mathcal{H} \oplus  \ell^2(\mathcal D_P)$:
	
	\[
	T_A=  \begin{bmatrix} 
		S & 0 & 0 & 0 & \dots \\
		A^*D_{P} & A & 0 & 0 & \dots \\
		0 & A^* & A & 0 & \dots\\
		0 & 0 & A^* & A & \dots\\
		\vdots & \vdots & \vdots & \ddots & \ddots \\
	\end{bmatrix} \quad \text{and}  \quad 
	V_0=  \begin{bmatrix} 
		P & 0 & 0 & 0 & \dots \\
		D_{P} & 0 & 0 & 0 & \dots \\
		0 & I & 0 & 0 & \dots\\
		0 & 0 & I & 0 & \dots\\
		\vdots & \vdots & \vdots & \ddots & \ddots \\
	\end{bmatrix}.
	\]
	The pair $(T_A, V_0)$ can also have the following $2 \times 2$ block matrix representation.
	\begin{equation}\label{T_A, V_0}
		T_A=  \begin{bmatrix} 
			S & 0  \\
			C & D \\
		\end{bmatrix}\quad \text{and} \quad V_0=  \begin{bmatrix} 
			P & 0 \\
			B & E \\
		\end{bmatrix},
	\end{equation}
	where $(D,E)$ is a commuting pair of operators on $\ell^2(\mathcal{D}_P)$ defined by
	\begin{equation}\label{D_E}
		D:=  \begin{bmatrix} 
			A & 0 & 0 &  \dots \\
			A^* & A & 0  & \dots \\
			0 & A^* & A &  \dots\\
			\vdots & \vdots & \ddots & \ddots \\
		\end{bmatrix} \quad \text{and} \quad 		
		E:=  \begin{bmatrix} 
			0 & 0 & 0 &  \dots \\
			I & 0 & 0 &  \dots \\
			0 & I & 0 &  \dots\\
			\vdots & \vdots & \ddots & \ddots \\
		\end{bmatrix}.
	\end{equation}
	Note that $(D,E)$ on $\ell^2(\mathcal{D}_P)$ is unitarily equivalent to $(T_{A+A^*z},T_z)$ on $H^2(\mathcal{D}_P)$, where $\omega(A) \leq 1$. Indeed, this is a standard model for a pure $\Gamma$-isometry, viz. Theorem \ref{thm:modelpure}. In this section, we mainly investigate the following questions.
	
	\medskip
	
\noindent $Q1$.  Is the pair $(T_A,V_0)$ always $\Gamma$-distinguished?

		\vspace{0.2cm}
	
\noindent $Q2$. Given a $\Gamma$-distinguished $\Gamma$-contraction $(S, P)$, is  the $\Gamma$-isometry $(T_A, V_0)$ always $\Gamma$-distinguished?
	
		\vspace{0.2cm}
	
\noindent $Q3$.  If $(S,P)$ and $(T_A,V_0)$ are both $\Gamma$-distinguished, then does every $\Gamma$-distinguished polynomial annihilating $(S,P)$ also annihilate $(T_A,V_0)?$

\medskip

About $Q1$, we would like to say that if this is to happen then $(S,P)$ is also $\Gamma$-distinguished, because, for every $f \in \mathbb{C}[z_1, z_2]$ we have
	$
	f(S, P)=P_\mathcal{H}f(T_A, V_0)|_\mathcal{H}.
	$
	Since every $\Gamma$-contraction $(S,P)$ dilates to such a $\Gamma$-isometry $(T_A,V_0)$ as in (\ref{D_E}), it is evident that the answer to this question is negative in general. There are $\Gamma$-contractions which are not $\Gamma$-distinguished as is shown in Example \ref{3.14}. We now turn to $Q2$ and $Q3$, for which we need the following lemma.

	\begin{lem}[\cite{Pal_Tomar}, Lemma 7.13]\label{lem1014}
		Let $\mathcal{H}_1$ and $\HS_2$ be Hilbert spaces. Let 
		\[
		(V_1, V_2)=\left(\begin{bmatrix}
			T_1 & 0 \\
			C_1 & D_1
			
		\end{bmatrix}, \begin{bmatrix}
			T_2 & 0 \\
			C_2 & D_2
		\end{bmatrix}\right)
		\] 
		be a commuting pair of operators on $\mathcal{H}_1 \oplus \mathcal{H}_2$. If $f(T_1, T_2)=0=g(D_1, D_2)$ for $f, g \in \C[z_1, z_2]$, then $f(V_1, V_2)g(V_1, V_2)=0$. Moreover, $(V_1, V_2)$ is a $\Gamma$-distinguished $\Gamma$-contraction if and only if both $(T_1, T_2)$ and $(D_1, D_2)$ are $\Gamma$-distinguished $\Gamma$-contractions. 
	\end{lem}

	\begin{rem} \label{7.2}
		With the notations as in (\ref{D_E}), we have by Lemma \ref{lem1014} that $(T_A,V_0)$ is $\Gamma$-distinguished if and only if $(S,P)$ and $(D,E)$ are $\Gamma$-distinguished.	
	\end{rem} 
	
		If $(S, P)$ is a $\Gamma$-distinguished $\Gamma$-contraction, then the above remark implies that its $\Gamma$-isometric dilation $(T_A, V_0)$ is $\Gamma$-distinguished if and only if the pure $\Gamma$-isometry $(D, E)$ is $\Gamma$-distinguished. Thus, the question if $(T_A, V_0)$ is $\Gamma$-distinguished boils down to whether or not the pure $\Gamma$-isometry $(D,E)$ is annihilated by a $\Gamma$-distinguished polynomial provided that $(S, P)$ is $\Gamma$-distinguished. We already have some results in this direction in Section \ref{sec05}. For example, Theorem  \ref{r(A)<1_(S,P)_distt.} provides a necessary and sufficient condition so that a pure $\Gamma$-isometry becomes $\Gamma$-distinguished under some additional hypothesis. Using the theory developed in Section \ref{sec05}, we have the following theorem. 
	
	\begin{thm}\label{thm:Char-f}
		Let $(S,P)$ be a $\Gamma$-distinguished $\Gamma$-contraction and $\dim\mathcal{D}_P < \infty$. Then the following are equivalent:
		\begin{enumerate}
			\item The minimal dilation $(T_A, V_0)$ is $\Gamma$-distinguished;
			\item $(A,0)$ is $\Gamma$-distinguished;
			\item $r(A)<1$.
		\end{enumerate}
	\end{thm}
	\begin{proof}
		
		By Remark \ref{7.2}, it suffices to prove that the pure $\Gamma$-isometry $(D,E)$ is $\Gamma$-distinguished. We also have that $(D,E)$ on $\ell^2(\mathcal{D}_P)$ is unitarily equivalent to the pure $\Gamma$-isometry $(T_{A+A^*z},T_z)$ on $H^2(\mathcal{D}_P)$. As $A^*$ is the fundamental operator of $(T_{A+A^*z}^*,T_z^*)$ and $\dim\mathcal{D}_{T_z^*}=\dim\mathcal{D}_P$ which is finite, Theorem \ref{r(A)<1_(S,P)_distt.} yields that the following statements are equivalent:
		\begin{itemize}
			\item[(a)]  $(T_{A+A^*z},T_z)$ is $\Gamma$-distinguished;
			\item[(b)]$(A^*,0)$ is $\Gamma$-distinguished;
			\item[(c)] $r(A^*)<1$.
		\end{itemize}
		The desired conclusion follows from Lemma 7.15 in \cite{Pal_Tomar}, which states that a $\Gamma$-contraction $(S_1, P_1)$ is $\Gamma$-distinguished if and only if $(S_1^*, P_1^*)$ is $\Gamma$-distinguished.		
	\end{proof}

	Let $(S, P)$ be a $\Gamma$-distinguished $\Gamma$-contraction on a Hilbert space $\mathcal{H}$. We show by an example that even if  $\dim\mathcal{D}_P$ is not finite, its minimal dilation $(T_A, V_0)$ (as specified in Theorem \ref{Tirtha-Pal}) can still be $\Gamma$-distinguished. 	
		\begin{eg}\label{7.4}
		Consider the space $\mathcal{H}=\ell^2(\mathbb{N})=
		\{(x_0, x_1, x_2, \dots): \ x_j \in \mathbb{C}, \ \  \sum_{j=0}^\infty|x_j|^2 < \infty \}$
		and take any $r \in (0,1)$. Consider the pair of commuting scalar contractions $(rI, rI)$ on $\mathcal{H}$, where $I$ is the identity operator on $\mathcal{H}$. We take the symmetrization of this pair and define $(S,P):=\pi(rI,rI)=(2rI, r^2I)$ on $\mathcal{H}$. Evidently, this is a $\Gamma$-contraction. Moreover, $(S, P)$ is annihilated by the polynomial $4z_2-z_1^2$ which is a $\Gamma$-distinguished polynomial as seen in Example \ref{3.4}. Hence, $(S,P)$ is a $\Gamma$-distinguished $\Gamma$-contraction. Let $A$ be the fundamental operator of this $\Gamma$-contraction. Since $D_P=\sqrt{1-r^4}I$, we have $\mathcal{D}_P=\mathcal{H}$ and so, $\mathcal{D}_P$ is not finite-dimensional. Also, $S-S^*P=2r(1-r^2)$ and $ D_PAD_P=(1-r^4)A$. From the uniqueness of the fundamental operator, it follows that $\displaystyle A=\frac{2r}{1+r^2}I$ and $\displaystyle \omega(A)=\frac{2r}{1+r^2} <1$ as $0<r<1$. Now we show that $(T_A,V_0)$ is $\Gamma$-distinguished. By Lemma \ref{7.2}, it suffices to show that the pure $\Gamma$-isometry $(D,E)$ is $\Gamma$-distinguished. Putting $\displaystyle a=\frac{2r}{1+r^2}$, we have
			\[
			D=  \begin{bmatrix} 
				A & 0 & 0 &  \dots \\
				A^* & A & 0  & \dots \\
				0 & A^* & A &  \dots\\
				\vdots & \vdots & \ddots & \ddots \\
			\end{bmatrix} =
			\begin{bmatrix} 
				aI & 0 & 0 &  \dots \\
				aI & aI & 0  & \dots \\
				0 & aI & aI &  \dots\\
				\vdots & \vdots & \ddots & \ddots \\
			\end{bmatrix}
			\quad 	\text{and} \quad 
			E=  \begin{bmatrix} 
				0 & 0 & 0 &  \dots \\
				I & 0 & 0 &  \dots \\
				0 & I & 0 &  \dots\\
				\vdots & \vdots & \ddots & \ddots \\
			\end{bmatrix}.
			\]
The polynomial $f(z_1, z_2)=z_1-az_2-a$ annihilates $(D,E)$. We show that $f$ is $\Gamma$-distinguished. Since $a<1,$  we have $\pi(a,0)=(a,0) \in Z(f) \cap \mathbb{G}_2$. For any $(\lambda_1+\lambda_2, \; \lambda_1\lambda_2) \in Z(f) \cap \partial \mathbb{G}_2$, we must have that either $\lambda_1$ or $\lambda_2$ is in $\mathbb{T}$. Without loss of generality, let $\lambda_1 \in \mathbb{T}$. Then 
			\[
			f(\lambda_1+\lambda_2,\; \lambda_1\lambda_2)=\lambda_1+\lambda_2-a\lambda_1\lambda_2-a=0 \quad \text{and so,} \quad \lambda_2=-\frac{\lambda_1-a}{1-a\lambda_1}.
			\]
			Since $z \mapsto -\dfrac{z-a}{1-az}$ with $a<1$ defines an automorphism of the unit disc which sends points of $\mathbb{T}$ onto $\mathbb{T}$, we have that $\lambda_2 \in \mathbb{T}$. Hence, $(\lambda_1, \lambda_2) \in \mathbb{T}^2$ which shows that $(\lambda_1+\lambda_2, \lambda_1\lambda_2) \in b\Gamma$. Thus, $Z(f) \cap \partial \mathbb{G}_2 \subseteq b \Gamma$. So, $(T_A, V_0)$ is a $\Gamma$-distinguished $\Gamma$-isometry and $\dim\mathcal{D}_P$ is not finite.
\qed 
	\end{eg}

	\begin{rem}
		Every $\Gamma$-distinguished polynomial that annihilates $(S,P)$ need not annihilate $(T_A,V_0)$ even if $(T_A,V_0)$ itself is $\Gamma$-distinguished.
		In the example above, the $\Gamma$-distinguished polynomial $4z_2-z_1^2$ annihilates $(S,P)$ but it does not annihilate $(T_A,V_0)$. Let if possible, $(D,E)$ be also annihilated by the same polynomial $4z_2-z_1^2$. Then $D^2=4E$ and a routine calculation shows that 
		\[
		D^2=  \begin{bmatrix} 
			a^2I & 0 & 0 & \dots \\
			* & a^2I & 0 & \dots \\
			* & * & a^2I & \dots\\
			\vdots & \vdots & \ddots & \ddots  \\
		\end{bmatrix}, \quad 
		4E=  \begin{bmatrix} 
			0 & 0 & 0 &  \dots \\
			4I & 0 & 0 &  \dots \\
			0 & 4I & 0 &  \dots\\
			\vdots & \vdots & \ddots & \ddots \\
		\end{bmatrix}.
		\]
	The equation $D^2=4E$ leads to $a^2=0$, which contradicts the fact that $0<a<1$.
	\end{rem} 
	
	We now focus on the following question: Is the $\Gamma$-isometry dilation $(T_A, V_0)$ of a $\Gamma$-distinguished $\Gamma$-contraction $(S, P)$ itself $\Gamma$-distinguished? Addressing this question requires a crucial observation which is essential for constructing a potential counter-example later. Before proceeding further, we recall that a commuting operator tuple $(T_1, \dotsc, T_n)$ is said to be \textit{algebraic} if it is annihilated by a polynomial $q \in \C[z_1, \dotsc, z_n]$.

	\begin{prop}\label{7.6}
		
		If $(T_A, V_0)$ is $\Gamma$-distinguished, then the following hold:
		\begin{enumerate}
			\item $(A,0)$ and $(S,P)$ are annihilated by the same polynomial that annihilates $(T_A, V_0);$
			\item the fundamental operator $A$ is algebraic;
			\item the spectrum of $A$ is finite that lies strictly inside $\mathbb{D}$.
		\end{enumerate}
	\end{prop}
	
	\begin{proof}
		Assume that $f(T_A,V_0)=0$ for some $\Gamma$-distinguished polynomial $f(z_1, z_2)$. It follows from Lemma \ref{lem1014} that $f(S,P)=0$ and $f(D,E)=0$. Using the block matrix form (\ref{D_E}), we rewrite the pair $(D,E)$ in the $2 \times 2$ block matrix form as
		\[
		D=  \begin{bmatrix} 
			A & 0  \\
			G & H \\
		\end{bmatrix} \quad \text{and} \quad
		E=  \begin{bmatrix} 
			0 & 0 \\
			K & L \\
		\end{bmatrix}
		\]
		with respect to $\mathcal{D}_P \oplus \ell^2(\mathcal{D}_P)$. Again by Lemma \ref{lem1014}, $f(A,0)=0$.  Since $f$ is $\Gamma$-distinguished, $f(z_1, z_2)$ cannot be of the form $z_2g(z_1, z_2)$ for any $g \in \mathbb{C}[z_1, z_2]$. Otherwise, $f(1,0)=0$ which implies that $(1,0)=\pi(1,0) \in Z(f) \cap \partial\mathbb{G}_2 \subseteq b\Gamma$. So, $(1,0) \in b\Gamma$ which is a contradiction. Moreover, $f(A, 0)=0$ implies that $f(z_1, z_2)$ cannot be of the form $a_0+ z_2g(z_1, z_2)$ for any $a_0 \in \C$ and $g \in \mathbb{C}[z_1, z_2]$. Thus, $f(z_1, z_2)$ has the following form: 
		\[
		f(z_1, z_2)=a_0+a_1z_1+\dots + a_nz_1^n+z_2g(z_1, z_2) 
		\]
		for some $g \in \mathbb{C}[z_1, z_2], \ n \in \mathbb{N}$ and $a_n \ne 0$. Hence,
		$
		p(z)=a_0+a_1z+\dots + a_nz^n
		$
		is a non-constant polynomial and $p(A)=f(A,0)=0$. By the spectral mapping theorem,
		$
		p(\sigma(A))=\sigma(p(A))=\{0\}.
		$
		Thus, $\sigma(A) \subset Z(p)$, which is a finite set. Since $A$ is the fundamental operator of $(S, P)$, we have that $r(A) \leq \omega(A) \leq 1$. If $r(A)=1$, then there exists $\lambda \in \mathbb{T} \cap \sigma(A)$. Again, it follows from spectral mapping theorem that 
		\begin{equation*}
			\begin{split}
				\{0\} =\sigma(f(A,0))
				=f(\sigma_T(A,0)) =f(\sigma(A)\times \{0\}).
			\end{split}
		\end{equation*}
		Hence, $f(\lambda,0)=0$. Since $(\lambda, 0) \in \partial \mathbb{D}^2$, we have $\pi(\lambda, 0) \in \partial \mathbb{G}_2$. Now, $f$ being $\Gamma$-distinguished implies that 
		$
		\pi(\lambda,0)=(\lambda,0) \in Z(f) \cap \partial \mathbb{G}_2 \subseteq b\Gamma,
		$
		which is a contradiction since $|p|=1$ for every $(s, p) \in b\Gamma$. Consequently, $r(A)<1$ and so, $\sigma(A)$ is a finite subset of $\mathbb{D}$.
	\end{proof}
	
		\begin{rem}\label{rem6.7}
		
		Following the proof of the above proposition, one can easily see that if $(T,0)$ is a $\Gamma$-distinguished $\Gamma$-contraction, then $\sigma(T)$ is a finite subset of $\mathbb{D}$. In particular, for the fundamental operator $A$, if the $\Gamma$-contraction $(A,0)$ is $\Gamma$-distinguished, then $\sigma(A)$ is a finite subset of $\D$.
		
	\end{rem}
	
Being armed with the above proposition, we are now in a position to show that the minimal $\Gamma$-isometric dilation $(T_A, V_0)$ of a $\Gamma$-distinguished $\Gamma$-contraction need not be $\Gamma$-distinguished.
	
	\begin{eg}\label{Shift_counter_example}
		Consider the right-shift operator $W$ on $\mathcal{H}=\ell^2(\mathbb{N})$ given by
		\[
		W(x_0,x_1, x_2, x_3, \dots)=(0, x_0, x_1, x_2, \dots)
		\]
		and take any $r \in (0,1)$. Consider the pair of commuting contractions $(rW, rW)$ on $\mathcal{H}$ and we take the symmetrization of this pair and define 
			$
			(S,P):=\pi(rW,rW)=(2rW, r^2W^2)$. Evidently, this is a $\Gamma$-contraction annihilated by the polynomial $f(z_1, z_2)=4z_2-z_1^2$, which is $\Gamma$-distinguished as shown in Example \ref{3.4}. Hence, $(S,P)$ is a $\Gamma$-distinguished pure $\Gamma$-contraction. Since $D_P^2=I-r^4W^*W^*WW=(1-r^4)I$, we have $D_P=\sqrt{1-r^4}I$ and $\mathcal{D}_P=\mathcal{H}$. We find the fundamental operator $A$ of $(S,P)$. Note that
			\[
			S-S^*P=2rW-2r^3W^*W^2=2r(1-r^2)W \ \ \text{and} \ \  D_P A D_P=(1-r^4)A.
			\]
			The uniqueness of fundamental operator gives that $\displaystyle A=\frac{2r}{1+r^2}W$ and $\displaystyle \omega(A)=\frac{2r}{1+r^2} <1$ as $0<r<1$. Suppose, if possible $(T_A, V_0)$ is $\Gamma$-distinguished. Proposition \ref{7.6} implies that $A=\omega(A)W$ is an algebraic operator with a finite spectrum. As the numerical radius $\omega(A)$ is a non-zero scalar, the spectral mapping theorem yields that $W$ also possesses a finite spectrum. This contradicts the fact that $\sigma(W)=\overline{\mathbb{D}}$. Hence, $(T_A, V_0)$ cannot be annihilated by a $\Gamma$-distinguished polynomial even though it is a $\Gamma$-isometric dilation of a $\Gamma$-distinguished $\Gamma$-contraction.
			
			\smallskip 
			
			Also, we mention here that $(S, P)$ admits a $\Gamma$-distinguished $\Gamma$-isometric dilation. To see this, let $[f_{ij}]_{1 \leq i, j \leq n}$ be a matricial polynomial. Since $P=S^2\slash 4$ and $S\slash 2$ is a contraction, we have	
			\begin{equation*}
				\begin{split}
					   \|[f_{ij}(S, S^2\slash 4)]_{i, j}\|
					 \leq   \max \{ \|[f_{ij}(2z, z^2)]_{i, j}\| \ :  z \in \DC \}
					 \leq   \max \{ \|[f_{ij}(z_1, z_2)]_{i, j}\| \ :  (z_1, z_2) \in Z(f) \cap \Gamma \}.
				\end{split}
			\end{equation*}
			Consequently, $Z(f) \cap \Gamma$ is a complete spectral set for $(S, P)$ and by Theorem \ref{thm814}, $(S, P)$ dilates to a $\Gamma$-distinguished $\Gamma$-isometry.  
		\qed 
	\end{eg} 
	
		We have seen in Proposition \ref{7.6} that a necessary condition for $(T_A, V_0)$ to be annihilated by a $\Gamma$-distinguished polynomial is the existence of a $\Gamma$-distinguished polynomial that annihilates both $(S,P)$ and $(A,0)$. We shall prove a partial converse to that. Indeed, we show the existence of a sequence of $\Gamma$-distinguished $\Gamma$-contractions that converges to $(T_A, V_0)$ in the strong operator topology when both $(S, P)$ and $(A, 0)$ are $\Gamma$-distinguished. For this purpose, we consider the sequence of commuting pair of operators $(T_n, V_n)$ on $\mathcal{H}\oplus \ell^2(\mathcal{D}_P)$ defined by 
		\begin{equation*}
			\begin{split}
				& T_n(x_0, x_1, x_2,  \dots)=(Sx_0, \; A^*D_Px_0+Ax_1, \; A^*x_1+Ax_2, \; \dots, \; A^*x_{n-1}+Ax_n, \; 0, \; 0, \dots);\\
				& V_n(x_0, x_1, x_2, \dots)=(Px_0, \; D_Px_0, \; x_1, \; \dots, \; x_{n-1},\; 0, \; 0, \dots).\\			
			\end{split}
		\end{equation*}
	The sequence $(T_n, V_n)$ is our candidate of $\Gamma$-contractions with desired properties.	With respect to the decomposition $\mathcal{H} \oplus \ell^2(\mathcal{D}_P)$, the $ 2 \times 2$ block matrix form of $(T_n, V_n)$ is given by
		\begin{equation}\label{T_n, V_n}
			T_n= \begin{bmatrix}
				S & 0\\
				C & D_n\\
			\end{bmatrix} \quad \mbox{and} \quad V_n=\begin{bmatrix}
				P & 0\\
				B & E_n\\
			\end{bmatrix},\\
		\end{equation}
		where the operators $C$ and $B$ are the same as defined in (\ref{T_A, V_0}). The operator pair $(D_n, E_n)$ on $\ell^2(\mathcal{D}_P)$ is defined as 
		\begin{equation}\label{D_n, E_n}
			D_n:=\begin{bmatrix}
				\hat{A}_n & 0\\
				0& 0\\
			\end{bmatrix} \quad \mbox{and} \quad E_n:=\begin{bmatrix}
				\hat{I}_n & 0\\
				0& 0\\
			\end{bmatrix}\\
		\end{equation}
		where the operator pair $(\hat{A}_n, \hat{I}_n)$ is given by 
		\begin{equation}\label{A_n, I_n}
			\hat{A_n}=  \begin{bmatrix} 
				A & 0 & 0 & \dotsc & 0 \\
				A^* & A & 0 & \dotsc & 0 \\
				0 & A^* & A & \dotsc & 0\\
				\dotsc & \dotsc & \dotsc & \dotsc & \dotsc\\
				0 & 0 & \dotsc & A^* & A \\
			\end{bmatrix}_{n\times n}, \quad
			\hat{I_n}=  \begin{bmatrix} 
				0 & 0 & 0 & \dotsc & 0 \\
				I & 0 & 0 & \dotsc & 0 \\
				0 & I & 0 & \dotsc & 0\\
				\dotsc & \dotsc & \dotsc & \dotsc & \dotsc\\
				0 & 0 & \dotsc & I & 0 \\
			\end{bmatrix}_{n\times n}.
		\end{equation}
Following the same notations as above, we show that $(T_n, V_n)$ is a $\Gamma$-contraction for every $n \in \mathbb{N}$. Our next result is a first step in this direction.
		
			\begin{lem}\label{A1.1}
			The pair $(\hat{A}_n, \hat{I}_n)$ on $\underset{n}{\oplus}\mathcal{D}_P$ is a $\Gamma$-contraction for every $n \in \mathbb{N}$.
		\end{lem}
		
		\begin{proof} We shall prove that the pair $(\hat{A}_n, \hat{I}_n)$ commutes, $\|\hat{A}_n\| \leq 2, \|\hat{I}_n\| \leq 1$ and there is a unique solution to the fundamental equation 
			\begin{equation}\label{eqn901}
				\hat{A}_n-\hat{A}_n^*\hat{I}_n=D_{\hat{I}_n}X_nD_{\hat{I}_n}
			\end{equation}
			for some $X_n \in B(\mathcal{D}_{\hat{I}_n})$ with $\omega(X_n) \leq 1$. Then Theorem \ref{thm_201} gives that $(\hat{A}_n, \hat{I}_n)$ is a $\Gamma$-contraction. Some routine computations give that 
			\[ 
			\hat{A}_n\hat{I}_n=  \begin{bmatrix} 
				0& 0& 0 & \dots & 0 \\
				A & 0 & 0 & \dots & 0 \\
				A^* & A & 0 & \dots & 0 \\
				\vdots & \vdots & \ddots & \ddots & \vdots\\
				0 & \dots & A^* & A & 0 \\
			\end{bmatrix}_{n\times n}=\hat{I}_n\hat{A}_n.
			\]
			Following the proof of Theorem 4.3 in \cite{Pal8}, we have $\|\hat{A}_n\| \leq 2$. It is easy to see that $\|\hat{I}_n\| =1$.  The LHS $\hat{A}_n- \hat{A}_n^*\hat{I}_n $ in (\ref{eqn901}) equals
			\begin{small} 
				\begin{align}\label{eqn902}
					\begin{split}
						\hat{A}_n- \begin{bmatrix} 
							A^*& A& 0 & \dots & 0 \\
							0 & A^* & A & \dots & 0 \\
							0 & 0 & A^* & \dots & \dots  \\
							\vdots & \vdots & \vdots & \ddots & A\\
							0 & 0 & 0 & \dots & A^* \\
						\end{bmatrix} \begin{bmatrix} 
							0 & 0 & 0 & \dots & 0 \\
							I & 0 & 0 & \dots & 0 \\
							0 & I & 0 & \dots & 0\\
							\vdots & \vdots & \ddots & \vdots & \vdots\\
							0 & 0 & \dots & I & 0 \\
						\end{bmatrix}
						&=\begin{bmatrix} 
							A & 0 & 0 & \dots & 0 \\
							A^* & A & 0 & \dots & 0 \\
							0 & A^* & A & \dots & 0\\
							\vdots & \vdots & \ddots & \ddots & \vdots\\
							0 & 0 & \dots & A^* & A \\
						\end{bmatrix}- \begin{bmatrix} 
							A& 0 & \dots & 0 & 0\\
							A^* & A & \dots & 0 & 0 \\
							0 & A^* & \dots & \dots & 0 \\
							\vdots & \vdots & \ddots & A & 0\\
							0 & 0 & \dots & A^* & 0\\
						\end{bmatrix} \\
						&=\begin{bmatrix}
							O_{n-1} & 0\\
							0& A\\
						\end{bmatrix},\\
					\end{split}
				\end{align}
			\end{small} 
			where $O_{n-1}$ denotes the zero block matrix of order $(n-1) \times (n-1)$. Next, we compute the defect operator and the  defect space corresponding to the operator $\hat{I}_n$. The operator $I-\hat{I}_n^*\hat{I}_n$ equals
			\begin{small} 
				\begin{equation*}
					\begin{split}
						I-   \begin{bmatrix} 
							0& I& 0 & \dots & 0 \\
							0 & 0 & I & \dots & 0 \\
							0 & 0 & 0 & \dots & 0  \\
							\vdots & \vdots & \vdots & \ddots & I\\
							0 & 0 & \dots & 0 & 0 \\
						\end{bmatrix} \begin{bmatrix} 
							0 & 0 & 0 & \dots & 0 \\
							I & 0 & 0 & \dots & 0 \\
							0 & I & 0 & \dots & 0\\
							\vdots & \vdots & \ddots & \vdots & \vdots\\
							0 & 0 & \dots & I & 0 \\
						\end{bmatrix}
						= \begin{bmatrix} 
							I& 0 & \dots & 0 &0\\
							0 & I & \dots & 0 &0 \\
							0 & 0 & \dots & \vdots &0 \\
							\vdots & \vdots & \ddots & I &0\\
							0 & 0 & \dots & 0 & I\\
						\end{bmatrix}- \begin{bmatrix} 
							I& 0 & \dots & 0 &0\\
							0 & I & \dots & 0 &0 \\
							0 & 0 & \dots & 0 &0 \\
							\vdots & \vdots & \ddots & I &0\\
							0 & 0 & \dots & 0 & 0\\
						\end{bmatrix}
						=\begin{bmatrix}
							O_{n-1} & 0\\
							0& I\\
						\end{bmatrix}. 
					\end{split}
				\end{equation*}
			\end{small} 
			Hence, $D_{\hat{I}_n}=(I-\hat{I}_n^*\hat{I}_n)^{1\slash 2}=\begin{bmatrix}
				O_{n-1} & 0\\
				0& I\\
			\end{bmatrix}$ and $\mathcal{D}_{\hat{I}_n}= \underbrace{0 \oplus 0  \oplus \dots  \oplus 0 \oplus  \mathcal{D}_P}_{n-times}$. Then the operator 
			\begin{equation*}
				X_n:\mathcal{D}_{\hat{I}_n} \to \mathcal{D}_{\hat{I}_n} \quad \text{given by} \quad 	X_n(0, 0, \dots, 0, x)=(0, 0, \dots, 0, Ax)
			\end{equation*}
			has numerical radius, $\omega(X_n)=\omega(A) \leq 1$. The matrix form of $X_n$ with respect to the decomposition  $\mathcal D_{\widehat I_n}\underbrace{=0 \oplus 0  \oplus \dots  \oplus 0 \oplus  \mathcal{D}_P}_{n-times}$ is 
			$
			X_n=\begin{bmatrix}
				O_{n-1} & 0\\
				0& A\\
			\end{bmatrix}.
			$
			It follows from (\ref{eqn902}) that
			\begin{equation*}
				\begin{split}
					D_{\hat{I}_n}X_nD_{\hat{I}_n} =\begin{bmatrix}
						O_{n-1} & 0\\
						0& I\\
					\end{bmatrix}\begin{bmatrix}
						O_{n-1} & 0\\
						0& A\\
					\end{bmatrix}\begin{bmatrix}
						O_{n-1} & 0\\
						0& I\\
					\end{bmatrix}
					=\begin{bmatrix}
						O_{n-1} & 0\\
						0& A\\
					\end{bmatrix}
					=\hat{A}_n- \hat{A}_n^*\hat{I}_n.\\
				\end{split}
			\end{equation*}
			Therefore, $(\hat{A}_n, \hat{I}_n)$ is a $\Gamma$-contraction on $\underset{n}{\oplus}\mathcal{D}_P$ for every  $n \in \mathbb{N}$.
		\end{proof}
		
	\begin{lem}\label{A1.2}
	The pair $(T_n, V_n)$ on $\mathcal{H}\oplus \ell^2(\mathcal{D}_P)$ is a $\Gamma$-contraction for every $n \in \mathbb{N}$.
\end{lem}

\begin{proof}
	We shall again use Theorem \ref{thm_201} to show that $(T_n, V_n)$ is indeed a $\Gamma$-contraction. We prove that $(T_n, V_n)$ is a commuting pair with $\|T_n\| \leq 2, \|V_n\| \leq 1$ and the fundamental equation  
	\begin{equation}\label{eqn903}
		T_n-T_n^*V_n=D_{V_n}Y_nD_{V_n}
	\end{equation}
	has a solution $Y_n$ with $\omega(Y_n) \leq 1$. Note that
	\begin{equation*}
		\begin{split}
			& T_nV_n(x_0, x_1, \dots) \\	&=T_n(Px_0, D_Px_0, x_1, \dots, x_{n-1}, 0,0, \dots) \\
			& = (SPx_0, \;A^*D_PPx_0+AD_Px_0, \;A^*D_Px_0+Ax_1, \;A^*x_1+Ax_2, \dots, \; A^*x_{n-2}+Ax_{n-1}, \; 0, \;0 \dots) \\
		\end{split}
	\end{equation*}
	and
	\begin{equation*}
		\begin{split}
			V_nT_n(x_0, x_1, \dots) &=V_n(Sx_0, \; A^*D_Px_0+Ax_1, \;A^*x_1+Ax_2, \dots, \;A^*x_{n-1}+Ax_n, \;0, \;0, \dots) \\
			& = (PSx_0, \;D_PSx_0, \; A^*D_Px_0+Ax_1, \; A^*x_1+Ax_2, \dots, \; A^*x_{n-2}+Ax_{n-1}, \; 0, \; 0,\dots).
		\end{split}
	\end{equation*}	
	Since $(S, P)$ is a $\Gamma$-contraction and $A$ is the fundamental operator of $(S,P)$, one can show that $D_P(A^*D_P+AD_P)=D_P^2S$. Thus, it follows that $A^*D_PP+AD_P=D_PS$ and hence, $T_nV_n=V_nT_n$.
	$ $
	\vspace{0.3cm}
	
	For any $x=(x_0, x_1, x_2, \dots) \in \mathcal{H} \oplus \ell^2(\mathcal{D}_P)$, it follows that 
	\begin{equation*}
		\begin{split}
		\|V_nx\|^2&=\|Px_0\|^2+\|D_Px_0\|^2+\|x_1\|^2+\dots +\|x_{n-1}\|^2=\|x_0\|^2+\|x_1\|^2+ \dotsc \|x_{n-1}\|^2 
			\leq \|x\|^2 \ \ \text{and} \\	 
	\|T_nx\|^2&=\|Sx_0\|^2+\|A^*D_Px_0+x_1\|^2+\|A^*x_1+x_2\|^2+\dots +\|A^*x_{n-1}+Ax_n\|^2 \\
			& \leq \|Sx_0\|^2+\|A^*D_Px_0+x_1\|^2+\overset{\infty}{\underset{j=1}{\sum}}\|A^*x_j+x_{j+1}\|^2\\
			&=\|T_A x\|^2.\\
		\end{split}
	\end{equation*}
	This gives that $\|V_n\| \leq 1$ and $\|T_n\| \leq \|T_A\| \leq 2$ for each $n$. The LHS in (\ref{eqn903}) is given by
	\begin{equation}\label{eqn904}
		\begin{split}
			T_n-T_n^*V_n = 
			\begin{bmatrix} 
				S & 0\\ 
				C & D_n
			\end{bmatrix}-\begin{bmatrix} 
				S^* & C^*\\ 
				0 & D_n^*
			\end{bmatrix}\begin{bmatrix} 
				P & 0\\ 
				B & E_n
			\end{bmatrix}
			=\begin{bmatrix} 
				S-S^*P-C^*B & -C^*E_n\\ 
				C-D_n^*B & D_n-D_n^*E_n
			\end{bmatrix}.
		\end{split}
	\end{equation}
	We further compute the operators appearing in the above $ 2 \times 2$ block matrix representation of $T_n-T_n^*V_n$. Since $C=	\begin{bmatrix} 
		D_PA & 0 &0 & \dots\\ 
	\end{bmatrix}^*$ and $B=\begin{bmatrix} 
		D_P & 0 &0 & \dots\\ 
	\end{bmatrix}^*$, we have
	\begin{equation*}
		\begin{split}
			C^*E_n& =\begin{bmatrix} 
				D_PA & 0 &0 & \dots\\ 
			\end{bmatrix}\begin{bmatrix} 
				0 & 0 & 0 & 0 & 0 &\dots \\
				I & 0 & 0 & 0 & 0 &\dots\\
				0 & I & 0 & 0 & 0 & \dots\\
				\vdots & \vdots & \ddots & \vdots & \vdots &\dots\\
				0& 0& 0& I & 0& \dots \\
				0& 0& 0& 0 & 0& \dots \\
				\vdots & \vdots & \vdots & \vdots & \vdots &\vdots\\
			\end{bmatrix}=0 \quad \text{and} \\
			D_n^*B&=\begin{bmatrix} 
				A^*& A& 0 & \dots & 0 & 0 & \dots  \\
				0 & A^* & A & \dots & 0 & 0 & \dots\\
				0 & 0 & A^* & \ddots & \dots  &\dots & \dots\\
				\vdots & \vdots & \vdots & \ddots & A & 0 & \dots\\
				0 & 0 & 0 & \dots & A^* & 0 & \dots\\
				0 & 0 & 0 & 0 & 0 & 0 & \dots\\
				\vdots & \vdots &\vdots &\vdots &\vdots &\vdots & \vdots\\
			\end{bmatrix}
			\begin{bmatrix}
				D_P \\
				0 \\ 0 \\ 0 \\ 0 \\ 0 \\ \dots \\
			\end{bmatrix}=\begin{bmatrix}
				A^*D_P \\
				0 \\ 0 \\ 0 \\ 0 \\ 0 \\ \dots \\
			\end{bmatrix}=C.
		\end{split}
	\end{equation*}
Lastly, we compute $D_n-D_n^*E_n$ which has the following $ 2 \times 2$ block representation with respect to the same decomposition as in the matrix form of $(T_n, V_n)$ given in (\ref{eqn904}).
	\begin{equation*}
		\begin{split}
			D_n-D_n^*E_n&=\begin{bmatrix} 
				\hat{A}_n &0\\
				0&0\\
			\end{bmatrix} - \begin{bmatrix} 
				\hat{A}_n^* &0\\
				0&0\\
			\end{bmatrix}\begin{bmatrix} 
				\hat{I}_n &0\\
				0&0\\
			\end{bmatrix}
			=\begin{bmatrix} 
				\hat{A}_n-\hat{A}_n^*\hat{I}_n &0\\
				0&0\\
			\end{bmatrix}.
		\end{split}
	\end{equation*}
	It follows from (\ref{eqn902}) and (\ref{eqn904}) that 
	\begin{equation*}
		\begin{split}
			T_n-T_n^*V_n &= 
			\begin{bmatrix} 
				0 & 0\\ 
				0 & D_n-D_n^*E_n
			\end{bmatrix}=[a_{ij}]_{i,j=0}^{\infty},
		\end{split}
	\end{equation*}    
	where $a_{n,n}=A$ and other block entries are zero.
	For any $x=(x_0, x_1, \dots) \in \mathcal{H} \oplus \ell^2(\mathcal{D}_P)$, we have \begin{equation*}
		\begin{split}
			(T_n-T_n^*V_n)(x_0, x_1, x_2, \dots)=(0,0, \dots, 0, Ax_n, 0,0, \dots),
		\end{split}
	\end{equation*}
	where $Ax_n$ is at the $(n+1)$-th position (counting from zero) and before that the $n$ entries are zero.
	
	\item We compute the defect operator and the  defect space for $V_n$.
	\begin{equation*}
		\begin{split}
			D_{{V_n}}^2&=I-V_n^*V_n
			=I- \begin{bmatrix} 
				P^* & B^*\\ 
				0 & E_n^*
			\end{bmatrix}\begin{bmatrix} 
				P & 0\\ 
				B & E_n
			\end{bmatrix}
			=I-\begin{bmatrix} 
				P^*P+B^*B & B^*E_n\\ 
				E_n^*B & E_n^*E_n
			\end{bmatrix}.
		\end{split}
	\end{equation*}
	We compute each block appearing in the above block matrix representation of $D_{V_n}^2$. Since $B^*B=D_{P}^2$, we must have that $P^*P+B^*B=I$. Note that
	\begin{equation*}
		\begin{split}
			B^*E_n=\begin{bmatrix} 
				D_P & 0 &0 & \dots\\ 
			\end{bmatrix}\begin{bmatrix} 
				0 & 0 & 0 & 0 & \dots &\dots \\
				I & 0 & 0 & 0 & \dots &\dots\\
				0 & I & 0 & 0 & \dots& \dots\\
				\vdots & \vdots & \ddots & \vdots & \vdots &\vdots\\
				0& 0& 0& I & 0& \dots \\
				0& 0& 0& 0 & 0& \dots \\
				\vdots & \vdots & \vdots & \vdots & \vdots &\vdots\\
			\end{bmatrix}=0.\\
		\end{split}
	\end{equation*}
	
	Next, we show that $E_n^*E_n=\begin{bmatrix}
		I_{n-1} & 0\\ 0& 0\\
	\end{bmatrix}$. We shall use the $2 \times 2$ block matrix representation of $E_n$.
	\begin{equation*}
		\begin{split}
			E_n^*E_n=\begin{bmatrix}
				\hat{I}_n^* & 0\\ 0& 0\\
			\end{bmatrix}\begin{bmatrix}
				\hat{I}_n & 0\\ 0& 0\\
			\end{bmatrix}=\begin{bmatrix}
				\hat{I}_n^* \hat{I}_n & 0\\ 0& 0\\
			\end{bmatrix}=\begin{bmatrix}
				I_{n-1} & 0\\ 0& 0\\
			\end{bmatrix}.
		\end{split}
	\end{equation*}
	Then  
	\begin{equation*}
		\begin{split}
			D_{{V_n}}^2
			=I-\begin{bmatrix} 
				I & 0\\ 
				0 & E_n^*E_n
			\end{bmatrix}
			=I-\begin{bmatrix} 
				I_n & 0\\ 
				0 & 0
			\end{bmatrix}
			=\begin{bmatrix} 
				O_{n} & 0\\ 
				0 & I
			\end{bmatrix}.
		\end{split}
	\end{equation*}
	Hence, $D_{V_n}=\begin{bmatrix} 
		O_{n} & 0\\ 
		0 & I
	\end{bmatrix}$ 
	and the defect space $\mathcal{D}_{V_n}=\underbrace{(0 \oplus \dots \oplus 0)}_{(n-times)} \oplus \mathcal{D}_P \oplus \mathcal{D}_P \oplus \dots$, i.e.,
	\begin{equation*}
		\mathcal{D}_{V_n}=\{(x_0, x_1, \dots) \in \mathcal{H}\oplus \ell^2(\mathcal{D}_P): x_0=x_1= \dots=x_{n-1}=0  \}.
	\end{equation*}
	Finally, we define our candidate for the fundamental operator of $(T_n, V_n)$. Consider the operator  
	\begin{equation*}
		Y_n: \mathcal{D}_{V_n} \to \mathcal{D}_{V_n} \quad \text{given by} \quad 			Y_n(0, 0, \dots, 0, x_{n}, x_{n+1}, x_{n+2}, \dots)= (0, 0, \dots, 0, Ax_{n}, 0, 0, \dots),
	\end{equation*}
	where $Ax_n$ is at the $(n+1)$-th position (counting from zero) and the first $n$ entries are zero. Since $\omega(Y_n)= \omega(A)$, we have $\omega(Y_n) \leq 1$. The RHS in (\ref{eqn903}) becomes
	\begin{equation*}
		\begin{split}
			D_{V_n}Y_nD_{V_n}(x_0, x_1, \dots)&=
			D_{V_n}Y_n(0,0, \dots, 0, x_n, x_{n+1}, x_{n+2}, \dots)\\
			&=D_{V_n}(0,0, \dots, 0, Ax_n, 0, 0, \dots)\\
			&=(0,0, \dots, 0, Ax_n, 0, 0, \dots)\\
			&=(T_n-T_n^*V_n)(x_0, x_1, \dots)\\
		\end{split}
	\end{equation*}
	for any $x=(x_0, x_1, \dots) \in \mathcal{H} \oplus \ell^2(\mathcal{D}_P)$. This shows that  
	\[
	D_{V_n}Y_nD_{V_n}= T_n-T_n^*V_n \ \mbox{with} \ \omega(Y_n) \leq 1 \ \mbox{and} \ Y_n \in B(\mathcal{D}_{V_n}).
	\]
	Therefore, $(T_n, V_n)$ on $\mathcal{H}\oplus \ell^2(\mathcal{D}_P)$ is a sequence of $\Gamma$-contractions. The proof is complete.
\end{proof}
		
We are now in a position to present the following approximation result.
		
	\begin{lem}\label{7.9}
		
		Let $(S,P)$ be a $\Gamma$-contraction on a Hilbert space $\mathcal{H}$ with fundamental operator $A$. Suppose both $(S,P)$ and $(A,0)$ are $\Gamma$-distinguished $\Gamma$-contractions. Then there is a sequence of $\Gamma$-distinguished $\Gamma$-contractions on $\mathcal{H}\oplus \ell^2(\mathcal{D}_P)$ which converges to $(T_A, V_0)$ in the strong operator topology.
	\end{lem} 

\begin{proof}
		
		Consider the sequence of commuting pair of operators $(T_n, V_n)$ on $\mathcal{H}\oplus \ell^2(\mathcal{D}_P)$ defined by 
		\begin{equation*}
			\begin{split}
				& T_n(x_0, x_1, x_2,  \dots)=(Sx_0, \; A^*D_Px_0+Ax_1, \; A^*x_1+Ax_2, \; \dots, \; A^*x_{n-1}+Ax_n, \; 0, \; 0, \dots);\\
				& V_n(x_0, x_1, x_2, \dots)=(Px_0, \; D_Px_0, \; x_1, \; \dots, \; x_{n-1},\; 0, \; 0, \dots).\\			
			\end{split}
		\end{equation*}
		It follows from Lemma \ref{A1.2} that each $(T_n, V_n)$ is a $\Gamma$-contraction on $\mathcal{H} \oplus \ell^2(\mathcal{D}_P)$. We show that each $(T_n, V_n)$ is $\Gamma$-distinguished. It is clear that each $(D_n, E_n)$ is $\Gamma$-distinguished if and only if each $(\hat{A}_n, \hat{I}_n)$ is $\Gamma$-distinguished. Since $(A, 0)$ is $\Gamma$-distinguished, a repeated application of Lemma \ref{lem1014} yields that each $(\hat{A}_n, \hat{I}_n)$ is $\Gamma$-distinguished. Again, using Lemma \ref{lem1014}, we can show that each $(T_n, V_n)$ is $\Gamma$-distinguished as $(S, P)$ and $(D_n, E_n)$ are $\Gamma$-distinguished for every $n \in \mathbb{N}$. Next, we show that the sequences $\{T_n\}_{n \in \mathbb{N}}$ and $\{V_n\}_{n \in \mathbb{N}}$ converge to $T_A$ and $V_0$ respectively in the strong operator topology. Given $x=(x_0, x_1, x_2, \dots) \in \mathcal{H}\oplus \ell^2(\mathcal{D}_P)$, we have that 
		$\displaystyle 
				\|T_Ax-T_nx\|^2= \overset{\infty}{\underset{j=n}{\sum}}\|A^*x_j+Ax_{j+1}\|^2 \to 0 $ as $n \to \infty$, because this sum is a tail of a convergent series with limit $\|T_Ax\|^2$. Similarly, we have that $\displaystyle 
				\|V_0x-V_nx\|^2= \overset{\infty}{\underset{j=n}{\sum}}\|x_j\|^2 \to 0$ as $n \to \infty$	being the tail of a convergent series with limit $\|x\|^2$.
	\end{proof}
	
		The $\Gamma$-isometric dilation $(T_A, V_0)$ of a $\Gamma$-distinguished $\Gamma$-contraction $(S, P)$ may or may not be $\Gamma$-distinguished. In Example \ref{Shift_counter_example}, we have already encountered one such case. It is worth finding out necessary and sufficient conditions such that the $\Gamma$-isometry $(T_A, V_0)$ becomes $\Gamma$-distinguished. The following result is a first step in this direction.	
	
		\begin{thm}	
		
		Let $(S,P)$ be a $\Gamma$-contraction acting on a Hilbert space $\mathcal H$. Then the following are equivalent.
		\begin{enumerate}
			\item $(T_A, V_0)$ is a $\Gamma$-distinguished $\Gamma$-isometry.
			
			\smallskip
				
			\item There is a $\Gamma$-distinguished polynomial that annihilates every $\Gamma$-contraction in the sequence $\{(T_n, V_n)\}_{n \in \mathbb{N}}$ as defined in \eqref{T_n, V_n}.
			
			\smallskip
					
			\item $(S,P)$ is $\Gamma$-distinguished and there is a $\Gamma$-distinguished polynomial annihilating each $\Gamma$-contraction in the sequence $\{(\hat{A}_n, \hat{I}_n)\}_{n \in \mathbb{N}}$ given in \eqref{A_n, I_n}.
		\end{enumerate}
		\end{thm}
	
		\begin{proof}
		We will prove that $(1) \implies (2) \implies (3) \implies (2) \implies (1)$.
		
			\vspace{0.2cm}
		
\noindent 	{$(1) \implies (2)$.} Let $(T_A, V_0)$ be a $\Gamma$-distinguished $\Gamma$-isometry. From the $2 \times 2$ block matrix form of $(T_n, V_n)$ as given in (\ref{T_n, V_n}), it follows that $(T_n, V_n)$ is $\Gamma$-distinguished if and only if $(S, P)$ and $(D_n, E_n)$ are $\Gamma$-distinguished. Since $(T_A, V_0)$ is $\Gamma$-distinguished, as a consequence of Lemma \ref{lem1014} and Proposition \ref{7.6}, it follows that the $\Gamma$-distinguished polynomial annihilating $(T_A, V_0)$ also annihilates $(S, P)$ and $(\hat{A}_n, \hat{I}_n)$. From the definition of $(D_n, E_n)$, it is clear that the same is true for $(D_n, E_n)$.  
		
\medskip
				
\noindent 	{$(2) \implies (3)$.} Let $f$ be a $\Gamma$-distinguished polynomial such that $f(T_n, V_n)=0$ for every $n \in\mathbb{N}$. Given the $2 \times 2$ block matrix form of $(T_n, V_n)$ in (\ref{T_n, V_n}), it follows from Lemma \ref{lem1014} that each $f(D_n, E_n)=0$ and $f(S,P)= 0$. Applying Lemma \ref{lem1014} on (\ref{D_n, E_n}), we have $f(\hat{A}_n, \hat{I}_n)=0$ for every $n \in \N$.
	
		\vspace{0.2cm}
		
\noindent 		{$(3) \implies (2)$.} Let $g$ and $h$ be $\Gamma$-distinguished polynomials such that
		$g(S,P)=0$ and $h(\hat{A}_n, \hat{I}_n)=0$ for every $n \in \N$. By Lemma \ref{lem1014} and (\ref{D_n, E_n}), $h(D_n, E_n)=0$ for all $n \in \N$. It follows from Lemma \ref{lem1014} that the $\Gamma$-distinguished polynomial $f(z_1, z_2)=g(z_1, z_2)h(z_1, z_2)$ annihilates each $(T_n, V_n)$.   
	
		\vspace{0.2cm}
		
	\noindent 	{$(2) \implies (1)$.} Let $f$ be a $\Gamma$-distinguished polynomial such that $f(T_n, V_n)=0$  for every $n \in\mathbb{N}$. It follows from Lemma \ref{7.9}  that $\{T_n\}_{n\in \N}$ and $\{V_n\}_{n \in \N}$ converge to $T_A$ and $V_0$ respectively in the strong operator topology. Moreover, $\|T_n\| \leq 2$ and $\|V_n\|\leq 1$ for each $n \in \N$. Consequently, $\{f(T_n, V_n)\}_{n \in \N}$ converges to $f(T_A, V_0)$ strongly and so, $f(T_A, V_0)=0$. The proof is complete.	
	\end{proof}

	The above result can be improved further if we assume an additional hypothesis on the fundamental operator of a $\Gamma$-contraction. In Proposition \ref{7.6}, we proved that if $(T_A, V_0)$ is $\Gamma$-distinguished, then $(A, 0)$ is $\Gamma$-distinguished as well. The subsequent results show that the converse is also true if one assumes that $A$ is hyponormal. 
	
		\begin{thm}\label{7.11}
		Assume that $(S,P)$ is a $\Gamma$-distinguished $\Gamma$-contraction and the fundamental operator $A$ of $(S,P)$ is normal. Then the $\Gamma$-isometry $(T_A, V_0)$ is $\Gamma$-distinguished if and only if $(A,0)$ is $\Gamma$-distinguished.
	\end{thm}
	
	\begin{proof} Assume that $(A,0)$ is $\Gamma$-distinguished. Following the proof of Proposition \ref{7.6}, one can deduce that there is a non-constant minimal polynomial $p(z)$ which annihilates $A$. Thus, 
	$
			p(A)=(\overline{\alpha}_1-A)(\overline{\alpha}_2-A) \dots (\overline{\alpha}_m-A)=0,
	$	
	where $\alpha_1, \dotsc, \alpha_m \in \C$. Define $X_j=(\overline{\alpha}_j-A)$ and $Y_j=X_j^*=(\alpha_j-A^*)$ for  $1 \leq j\leq m$. For $Q_j \in \{X_j, Y_j \}$ with  $1 \leq j \leq m$, we show that $Q_1\dots Q_m=0$. To prove this, we consider the operators on $\mathcal{D}_P$  given by
		\[ 
		R_j= \left\{
		\begin{array}{ll}
			X_j &  \mbox{if} \ Q_j=Y_j \\
			Y_j & \mbox{if} \ Q_j=X_j\\
		\end{array} \quad (1 \leq j \leq m).
		\right. 
		\]
		Since $A$ is normal, we have that $X_jY_i=Y_iX_j$ for $1 \leq i, j \leq m$. Then
			\begin{equation*}
			\begin{split}
				0=p(A)p(A)^*
				=(X_1\dots X_m)(Y_m\dots Y_1)
				=(Q_1\dots Q_m)(R_m\dots R_1)
				=(Q_1\dots Q_m)(Q_1\dots Q_m)^*,\\
			\end{split}
		\end{equation*}
		and hence, $Q_1\dots Q_m=0$. From Remark \ref{rem6.7}, it follows that $r(A)<1$ and so, $|{\alpha}_j| <1$ for each $j$.  Hence, the polynomial $\overline{\alpha}_j+\alpha_jz_2-z_1$ is $\Gamma$-distinguished for each $j$ as shown in Example \ref{7.4}. This means that the polynomial defined by 
	$\displaystyle
			f(z_1, z_2)=\overset{m}{\underset{j=1}{\prod}}(\overline{\alpha}_j+\alpha_jz_2-z_1)
	$
		is a $\Gamma$-distinguished polynomial being a finite product of such polynomials. Next, we prove that $f(z_1, z_2)$ annihilates $(T_{A+A^*z},T_z)$ on $H^2(\mathcal{D}_P)$. Indeed, we have
		\begin{equation*}
			\begin{split}
				f(A+A^*z,z)=\overset{m}{\underset{j=1}{\prod}}\bigg(\overline{\alpha}_j+\alpha_jz-(A+A^*z)\bigg)
				&=\overset{m}{\underset{j=1}{\prod}}\bigg((\overline{\alpha}_j-A)+(\alpha_j-A^*)z\bigg)\\
				&=\overset{m}{\underset{j=1}{\prod}}(X_j+Y_jz)\\
				&=\overset{m}{\underset{j=1}{\prod}}X_j+z\overset{m}{\underset{j=1}{\prod}}Q_{j}^{(1)}+z^2\overset{m}{\underset{j=1}{\prod}}Q_j^{(2)}+ \dots + z^m\overset{m}{\underset{j=1}{\prod}}Y_j,
			\end{split}
		\end{equation*}
		where $Q_j^{(k)}$ is either $X_j$ or $Y_j$. Hence, $\overset{m}{\underset{j=1}{\prod}}Q_j^{(k)}=0$ for each $k$. Consequently, $f$ annihilates $(T_{A+A^*z},T_z)$ and since $(D,E)$ is unitarily equivalent to $(T_{A+A^*z},T_z)$, we have $f(D,E)=0$. Now, it follows from Remark \ref{7.2} that $(T_A, V_0)$ is $\Gamma$-distinguished. The converse is a direct consequence of Remark \ref{7.2}. The proof is now complete.
	\end{proof}
	
		Recall that an operator $T$ defined on a Hilbert space $\mathcal{H}$ is said to be \textit{hyponormal} if $T^*T-TT^* \geq 0$, or equivalently $\|T^*x\| \leq \|Tx\|$ for every $x$ in $\mathcal{H}$.  Next, we prove that the above result holds when the fundamental operator of a $\Gamma$-distinguished $\Gamma$-contraction is assumed to be hyponormal. 
		
	\begin{thm}\label{7.12}
		Assume that $(S,P)$ is a $\Gamma$-distinguished $\Gamma$-contraction and the fundamental operator $A$ of $(S,P)$ is hyponormal. Then $(T_A, V_0)$ is $\Gamma$-distinguished if and only if  $(A,0)$ is $\Gamma$-distinguished.
	\end{thm}
		\begin{proof} A hyponormal operator annihilated by a polynomial is normal. This fact is an easy consequence of Corollary 2 in \cite{Stampfli}. The desired conclusion follows from Theorem \ref{7.11}. 
	\end{proof}
	
		We conclude this section with the following corollary. 
	
		\begin{cor}
		Let $(S,P)$ be a normal $\Gamma$-distinguished $\Gamma$-contraction with $\|P\|<1$ and let $A$ be its fundamental operator. Then $(T_A, V_0)$ is $\Gamma$-distinguished if and only if $(A,0)$ is $\Gamma$-distinguished.
	\end{cor}
	
		\begin{proof}
		The defect operator $D_P$ is invertible since $\|P\|<1$. Hence, the fundamental operator $A$ of $(S, P)$ is given by $A=D_P^{-1}(S-S^*P)D_P^{-1}$. Evidently, $A$ is normal as $S$ and $P$ are normal operators. The desired conclusion now follows from Theorem \ref{7.12}.
	\end{proof}

		\section{Decomposition of $\Gamma$-unitaries and pure $\Gamma$-isometries  annihilated by distinguished  polynomials}\label{sec_decomp}
	
	\vspace{0.1cm}
	
\noindent Every $\Gamma$-isometry admits a Wold type decomposition, e.g., see \cite{AglerII16}. Indeed, if $(T,V)$ is a $\Gamma$-isometry acting on a Hilbert space $\mathcal{H}$, then there is an orthogonal decomposition of $\mathcal{H}$ into closed joint reducing subspaces $\mathcal{H}_u$ and $\mathcal{H}_p$ such that$(T|_{\mathcal{H}_u}, V|_{\mathcal{H}_u})$ is a $\Gamma$-unitary and $(T|_{\mathcal{H}_p}, V|_{\mathcal{H}_p})$ is a pure $\Gamma$-isometry. It follows that if $(T, V)$ is $\Gamma$-distinguished, then so are $(T|_{\mathcal{H}_u}, V|_{\mathcal{H}_u})$ and $(T|_{\mathcal{H}_p}, V|_{\mathcal{H}_p})$. Naturally, the polynomial that annihilates $(T, V)$ also annihilates these two pairs. In other words, we have the following.
	
		\begin{lem} 
		
		Let $(T, V)$ be a $\Gamma$-distinguished $\Gamma$-isometry on $\mathcal{H}$. Then  there is an orthogonal decomposition of $\mathcal{H}$ into closed joint reducing subspaces $\mathcal{H}_u$ and $\mathcal{H}_p$ such that $(T|_{\mathcal{H}_u}, V|_{\mathcal{H}_u})$ is a $\Gamma$-distinguished $\Gamma$-unitary and  $(T|_{\mathcal{H}_p}, V|_{\mathcal{H}_p})$ is a $\Gamma$-distinguished pure $\Gamma$-isometry.
	\end{lem}

The preceding lemma shows that any $\Gamma$-distinguished $\Gamma$-isometry splits into a $\Gamma$-distinguished pure $\Gamma$-isometry and a $\Gamma$-distinguished $\Gamma$-unitary. We now go a step further and provide decomposition results for a subclass of such operator pairs, namely those which are annihilated by polynomials whose zero sets lie in $\G \cup b\Gamma \cup \pi(\mathbb{E}^2)$. This decomposition is motivated by Theorem 2.1 in \cite{AglerKneseMcCarthy2}, where it was proved that a commuting pair of pure isometries annihilated by a polynomial $q \in \C[z_1, z_2]$ with $Z(q) \subseteq \D^2 \cup \T^2 \cup \mathbb{E}^2$ can be expressed nearly as a direct sum corresponding to the irreducible factors of $q$.

\medskip Recall from \cite{AglerKneseMcCarthy2} that a polynomial $q \in \C[z_1, z_2]$ is said to be \textit{inner toral} if $Z(q) \subseteq \D^2 \cup \T^2 \cup \mathbb{E}^2$. To go parallel with this, we say that a polynomial $p \in \C[z_1, z_2]$ is \textit{distinguished} if  $Z(p) \subseteq \G \cup b\Gamma \cup \pi(\mathbb{E}^2)$, and a $\Gamma$-contraction $(S, P)$ is said to be \textit{distinguished} if it is annihilated by a distinguished polynomial. Evidently, multiplying a distinguished polynomial with the $\Gamma$-distinguished polynomial $p(z_1, z_2) = z_1$ gives a $\Gamma$-distinguished polynomial. Also, these two classes of polynomials are not same. For example, the polynomial $p(z_1, z_2) = z_1(z_2 - 1)$ is $\Gamma$-distinguished but not distinguished. Indeed, $(5/2, 1) \in Z(p)$, which can only arise as the symmetrization of the pairs $(2, 1/2)$ and $(1/2, 2)$. Since neither of these pairs belongs to $\D^2 \cup \T^2 \cup \mathbb{E}^2$, we have that $(5/2, 1) \notin \G \cup b\Gamma \cup \pi(\mathbb{E}^2)$. The following lemma clarifies the fact that the class of $\Gamma$-distinguished $\Gamma$-contractions is bigger than that of distinguished $\Gamma$-contractions.
	
\begin{lem}\label{5.2}
		Every distinguished $\Gamma$-contraction is $\Gamma$-distinguished. 
	\end{lem}

	\begin{proof}
		Let $(S, P)$ be a $\Gamma$-contraction acting on a Hilbert space $\HS$ and let $q \in \C[z_1, z_2]$ be a distinguished polynomial that annihilates $(S, P)$. Let us define $p(z_1, z_2)=z_1q(z_1, z_2)$. Evidently, $(0, 0) \in Z(p) \cap \G$. Since $q(z_1, z_2)$ and $g(z_1, z_2)=z_1$ are distinguished and $\Gamma$-distinguished polynomials respectively, we have that $Z(p) \cap \partial \Gamma= Z(p) \cap b\Gamma$. The proof is now complete.   	 
	\end{proof}

A distinguished pure $\Gamma$-isometry has a square-free minimal annihilating polynomial as shown below, which follows from Theorem 3.16 in \cite{Pal_Tomar}.
	
	\begin{lem}\label{5.4}
		Let $(S,P)$ be a distinguished pure $\Gamma$-isometry. Then there is a square-free distinguished polynomial $q$ that annihilates $(S,P) $. Moreover, if $p$ is any polynomial that annihilates $(S,P)$ then $q$ divides $p$.
	\end{lem} 
	
	\begin{proof}
		Since $(S, P)$ is algebraic, Theorem 3.16 in \cite{Pal_Tomar} ensures the existence of a square-free minimal polynomial $q$ annihilating $(S,P)$. Let $f$ be a distinguished polynomial such that $f(S,P)=0$. Then $q$ divides $f$ and thus, $ Z(q) \subseteq Z(f) \subseteq \G \cup b\Gamma \cup \pi(\mathbb{E}^2)$, which completes the proof.
	\end{proof}
	
We now turn to the first part of our decomposition theorems, focusing on distinguished $\Gamma$-unitaries. To do so, we need the following result, which will be used throughout this section.

	\begin{lem}\label{8.1}
		Let $(T,U)$ be a $\Gamma$-unitary acting on a Hilbert space $\mathcal{K}$ annihilated by a distinguished polynomial $q$. Let $q_1$ and $q_2$ be distinct factors of $q$ such that $q=q_1q_2$. Then
		\[
		q_1(T,U)\mathcal{K} \perp q_2(T,U)\mathcal{K}.
		\]
	\end{lem}
	
	\begin{proof}
		Since $(T,U)$ is a $\Gamma$-unitary on $\mathcal{K}$, there is a commuting pair $\mathcal{U}=(U_1,U_2)$ of unitaries on $\mathcal{K}$ such that $\pi(U_1,U_2)=(T,U)$. Since $q$ is a distinguished polynomial, the polynomial $p=q \circ \pi$ is inner toral and $p(\mathcal{U})=0$. Let $p_i=q_i\circ \pi$ for $i=1, 2$ so that $p=p_1p_2$. Since $p_1$ is also inner toral, we have that 
$			z_1^nz_2^m\overline{p_1(1\slash{\overline{z}_1}, 1\slash{\overline{z}_2})}=\alpha \; p_1(z_1, z_2),
$
		where $\alpha \in \mathbb{T}$ and $(n,m)$ is the degree of $p_1$. In fact, we may assume $\alpha=1$ by replacing $p_1$ with an appropriate constant multiple (see \cite{Knese9} and Lemma 2.2 in \cite{AglerKneseMcCarthy2} for further details). Consequently, 
		\begin{equation*}
			p_1(z_1, z_2)=\overset{n}{\underset{i=0}{\sum}}\overset{m}{\underset{j=0}{\sum}}a_{ij}z_1^iz_2^j
		\end{equation*}
		and  $\overline{a}_{ij}=a_{(n-i)(m-j)}$. This can be used to show that 
		$
		p_1(\mathcal{U})^*U_1^nU_2^m=p_1(\mathcal{U}),
		$
		as $U_1, U_2$ are commuting unitaries. Consequently, 
		$p_1(\mathcal{U})^*p_2(\mathcal{U})=U_1^{*n}U_2^{*m}p_1(\mathcal{U})p_2(\mathcal{U})=U_1^{*n}U_2^{*m}p(\mathcal{U})=0$. Thus,		$
			\langle q_1(T,U)h_1, q_2(T,U)h_2\rangle = \langle p_1(\mathcal{U})h_1, p_2(\mathcal{U})h_2 \rangle = \langle h_1, p_1(\mathcal{U})^*p_2(\mathcal{U})h_2 \rangle = 0$ for every $h_1, h_2 \in \mathcal{K}$.
		Thus, we have proved that $q_1(T,U)\mathcal{K} \perp q_2(T,U)\mathcal{K}$.
	\end{proof}
	
We are now in a position to present our decomposition theorem for a $\Gamma$-unitary annihilated by a distinguished polynomial $q \in \C[z_1, z_2]$. To explain the idea of the proof, we first prove the decomposition result for $q=q_1q_2$, where $q_1, q_2$ are distinct factors of $q$. 

\smallskip 

For the sake of brevity, we fix the following notations for the rest of this section:  A \textit{$q$-$\Gamma$-contraction} means a $\Gamma$-contraction annihilated by a polynomial $q(z_1, z_2)$. For a commuting pair $\Sigma=(T_1, T_2)$ acting on a Hilbert space $\mathcal{H}$ and a polynomial $q$, let $q(\Sigma)$ and $q(\Sigma)^*$ denote the operators $q(T_1, T_2)$ and $(q(T_1, T_2))^*$ respectively. For a closed joint invariant subspace $\mathcal{L} \subseteq \mathcal{H}$ of $T_1, T_2$, we denote $(T_1|_{\mathcal{L}}, T_2|_{\mathcal{L}})$ by $\Sigma|_{\mathcal{L}}$.

	\begin{prop}\label{8.2}
		
		Let $\Sigma=(T,U)$ be a $\Gamma$-unitary on $\mathcal{K}$ annihilated by a distinguished polynomial $q$ and let $q_1, q_2$ be distinct factors of $q$ such that $q=q_1q_2$. Then there exist closed joint reducing subspaces $\mathcal{K}_1,\mathcal{K}_2$ of $\mathcal{K}$ such that 
		\begin{enumerate}
			\item $\mathcal{K}=\mathcal{K}_1 \oplus \mathcal{K}_2$ and
			\item each $\Sigma|_{\mathcal{K}_j}$ is a $\Gamma$-unitary annihilated by $q_j$ for $j=1,2$.
		\end{enumerate}
		
	\end{prop}
	
	\begin{proof}
		It follows from Lemma \ref{8.1} that 
		$
		q_1(\Sigma)\mathcal{K} \perp q_2(\Sigma)\mathcal{K}.
		$
		The space $\mathcal{L}_1=\overline{q_2(\Sigma)\mathcal{K}}$ is a closed joint reducing subspace of $\mathcal{K}$ and $\Sigma|_{\mathcal{L}_1}$ is a $q_1$- $\Gamma$-unitary. The space $		\mathcal{K}_2=\overline{q_1(\Sigma)\mathcal{K}}$ is orthogonal to $\mathcal{L}_1$ and
		is a joint reducing subspace too. Hence, $\Sigma|_{\mathcal{K}_2}$ is a $q_2$-$\Gamma$-unitary. Thus, we have the following orthogonal decomposition of $\mathcal{H}$ into the closed joint reducing subspaces: 
		\[
		\mathcal{K}=\mathcal{L}_1 \oplus \mathcal{K}_2 \oplus \mathcal{L}',
		\]
		where, $\mathcal{L}'=(\mbox{Ran}\;q_1(\Sigma)\oplus\mbox{Ran}\;q_2(\Sigma))^\perp=\mbox{Ker}\;q_1(\Sigma)^*\cap \mbox{Ker}\;q_2(\Sigma)^*=\mbox{Ker}\;q_1(\Sigma)\cap \mbox{Ker}\;q_2(\Sigma)$, where the last equality holds since $q_1(\Sigma), q_2(\Sigma)$ are normal operators. Since $q_1\circ \pi$ and $q_2\circ \pi$ are inner toral, we can prove the following (similar to the proof of Lemma \ref{8.1}). 
		\begin{enumerate}
			\item  $q_1\circ \pi(U_1, U_2)^*=U_1^{*j}U_2^{*k}q_1\circ\pi(U_1, U_2)$
			\item $q_2\circ \pi(U_1, U_2)^*=U_1^{*n}U_2^{*m}q_2\circ\pi(U_1, U_2)$, 
		\end{enumerate}
		where $U_1,U_2$ are commuting unitaries on $\mathcal K$ such that $\pi(U_1,U_2)=(T,U)$ and $j,k,n,m \geq 0$. Consequently, $q_1(\Sigma)^*=U_1^{*j}U_2^{*k}q_1(\Sigma)$ and  $q_2(\Sigma)^*=U_1^{*n}U_2^{*m}q_2(\Sigma)$. This shows that for any $x \in \mathcal{L}'$, we have $ U_1^{*j}U_2^{*k}q_1(\Sigma)x=q_1(\Sigma)^*x=0$.
		Hence, $q_1(\Sigma)|_{\mathcal{L}'}=0$. The space $\mathcal{K}_1=\mathcal{L}_1 \oplus \mathcal{L}'$ is a closed joint reducing subspace and $\Sigma|_{\mathcal{K}_1}$ is a $q_1$-$\Gamma$-unitary. The desired conclusion follows.
	\end{proof}
	
		The following generalized version of our previous result follows from mathematical induction and Proposition \ref{8.2}. 
	
	\begin{thm}
		Let $\Sigma=(T,U)$ be a $\Gamma$-unitary acting on a Hilbert space $\mathcal{K}$ annihilated by a distinguished polynomial $q$ and let $q_1, \dots, q_N$ be the distinct irreducible factors of $q$. Then there is an orthogonal decomposition of $\mathcal{K}$ into closed joint reducing subspaces $\mathcal{K}_1, \dotsc, \mathcal{K}_N$ such that each $\Sigma|_{\mathcal{K}_j}$ is a $\Gamma$-unitary annihilated by $q_j$.
			\end{thm}

Having established the above decomposition theorem for distinguished $\Gamma$-unitaries, we next obtain an analogous decomposition for distinguished pure $\Gamma$-isometries, which is the main result of this section. Recall from Theorem \ref{thm:modelpure} that every pure $\Gamma$-isometry $(T, V)$ is unitarily equivalent to the restriction of the $\Gamma$-unitary $(M_\phi, M_z)$ on $L^2(\mathcal{D}_{V^*})$ to the joint invariant subspace $H^2(\mathcal{D}_{V^*})$. 

	\begin{thm}\label{8.4}		
		Let $\Sigma=(T, V)$ be a pure $\Gamma$-isometry on $\mathcal{H}$. Let $q$ be a distinguished polynomial that annihilates $(T, V)$ and let $q_1, \dots, q_N$ be the distinct irreducible factors of $q$. Then there exist $(N+1)$ closed orthogonal disjoint subspaces $\mathcal{H}_1, \dotsc, \mathcal{H}_N$ of $\mathcal{H}$ that are invariant under both $T$ and $V$ such that $\mathcal{H}=\mathcal{H}_1 \oplus \dotsc \oplus \mathcal{H}_N \oplus \mathcal{H}'$ and each $\displaystyle \Sigma|_{\mathcal{H}_j}$ is a pure $\Gamma$-isometry annihilated by $q_j$, where $\HS'=\HS \ominus (\HS_1 \oplus \dotsc \oplus \HS_N)$. Moreover, each $\mathcal{H}_j=\overline{r_j(\Sigma)\mathcal{H}}$ for the polynomial $r_j=q \slash q_j$.  	
	\end{thm}
	
	\begin{proof}
		For $q=q_1q_2\dotsc q_N$, let $p=q_2q_3\dots q_N$. We will show that $\mathcal{H}$ has a closed joint invariant subspace restricted to which $\Sigma$ can be written as the direct sum of a $q_1$-$\Gamma$-isometry and a $p$-$\Gamma$-isometry. The desired conclusion then follows from mathematical induction. By Theorem \ref{thm:modelpure}, $(T, V)$ is unitarily equivalent to the restriction of the $\Gamma$-unitary $\mathcal{U}=(M_\phi, M_z)$ acting on $\mathcal{K}=L^2(\mathcal{D}_{V^*})$ to the joint invariant subspace $H^2(\mathcal{D}_{V^*})$. Without loss of generality, we assume that $\HS=H^2(\mathcal{D}_{V^*})$ and $(T, V)=(M_\phi|_\HS, M_z|_\HS)$. Since 
		$(M_\phi, M_z)$ is (up to unitary equivalence) the minimal normal extension of $(T, V)$, the proof of Proposition 7.1 in \cite{Pal_Tomar} guarantees that $q(\mathcal{U})=0$. By Proposition \ref{8.2}, there is an orthogonal decomposition of $\mathcal{K}$ into disjoint closed subspaces $\mathcal{K}_1, \mathcal{K}_2$ reducing both $M_\phi, M_z$ such that $\mathcal{U}|_{\mathcal{K}_1}$ is a $q_1$-$\Gamma$-unitary and $\mathcal{U}|_{\mathcal{K}_2}$ is a $p$-$\Gamma$-unitary. The subspaces are given by 
		\[
		\mathcal{K}_1=\overline{p(\mathcal{U})\mathcal{K}} \oplus \bigg[\mbox{Ker} \; q_1(\mathcal{U})\cap  \mbox{Ker} \; p(\mathcal{U})\bigg] \quad \mbox{and} \quad \mathcal{K}_2=\overline{q_1(\mathcal{U})\mathcal{K}}.
		\]
		The subspaces of $\mathcal{H}$ defined by
		$			
		\mathcal{H}_1=\overline{p(\Sigma)\mathcal{H}}$ and $\mathcal{H}_1'=\overline{q_1(\Sigma)\mathcal{H}}
$ are invariant under $T$ and $V$. Since $p(\Sigma)\mathcal{H}=p(\mathcal{U})\mathcal{H}\subseteq \mathcal{K}_1$ and $q_1(\Sigma)\mathcal{H}=q_1(\mathcal{U})\mathcal{H}\subseteq \mathcal{K}_2$, we have that $\mathcal{H}_1$ and $\mathcal{H}_1'$ are subspaces of $\mathcal{K}_1$ and $\mathcal{K}_2$ respectively. Hence, $\mathcal{H}_1$ and $\mathcal{H}_1'$ are disjoint orthogonal spaces such that $\Sigma|_{\mathcal{H}_1}$ and $\Sigma|_{\mathcal{H}_1'}$ are pure $\Gamma$-isometries annihilated by $q_1$ and $p$ respectively. The pair $\Sigma_2=(T_2, V_2)=(T|_{\mathcal{H}_1'}, V|_{\mathcal{H}_1'})$ is a pure $\Gamma$-isometry and $p(T_2, V_2)=0$. Again, applying Theorem \ref{thm:modelpure}, the model theorem for a pure $\Gamma$-isometry, we have that $(T_2, V_2)$ is unitarily equivalent to $(T_{\phi_2}, T_z)$ on $H^2(\mathcal{D}_{V_2^*})$ and $p(T_{\phi_2}, T_z)=0$. This yields that the $\Gamma$-unitary $\mathcal{U}_2=(M_{\phi_2}, M_z)$ on $\mathcal{K}_2=L^2(\mathcal{D}_{V_2^*})$ is annihilated by $p=q_2\dotsc q_N$. Let $p_2=q_3\dotsc q_N$. Proposition \ref{8.2} shows that $\mathcal{K}_2$ can be written as an orthogonal decomposition of reducing subspaces $\mathcal{K}_{21}$ and $\mathcal{K}_{22}$ such that $\mathcal{U}_2|_{\mathcal{K}_{21}}$ is a $q_2$-$\Gamma$-unitary and $\mathcal{U}_2|_{\mathcal{K}_{22}}$ is a $p_2$-$\Gamma$-unitary. Moreover, the structures of these two subspaces are as follows.
		\[
		\mathcal{K}_{21}=\overline{p_2(\mathcal{U}_2)\mathcal{K}_2} \oplus \bigg[\mbox{Ker} \; q_2(\mathcal{U}_2)\cap  \mbox{Ker} \; p_2(\mathcal{U}_2)\bigg] \quad \mbox{and} \quad \mathcal{K}_{22}=\overline{q_2(\mathcal{U}_2)\mathcal{K}_2}.
		\]
		One can refer to the proof of Proposition \ref{8.2} to obtain the above two subspaces. Now, the spaces $	\mathcal{H}_2=\overline{p_2(\Sigma)\mathcal{H}_1'}$ and $\mathcal{H}_2'=\overline{q_2(\Sigma)\mathcal{H}_1'}$ are (unitarily equivalent to) closed joint subspaces of $\mathcal{K}_{21}$ and $\mathcal{K}_{22}$ respectively which are both invariant under $T$ and $V$. Thus, $\mathcal{H}_2$ and $\mathcal{H}_2'$ are orthogonal and it is not difficult to see that $\Sigma_2|_{\mathcal{H}_2}$ and $\Sigma_2|_{\mathcal{H}_2'}$ are pure $\Gamma$-isometries annihilated by $q_2$ and $p_2$ respectively. Needless to mention that $\mathcal{H}_2$ is the same as $\overline{q_1(\Sigma)p_2(\Sigma)\mathcal{H}}$. Thus, there exist closed disjoint orthogonal subspaces $\mathcal{H}_1$ and $\mathcal{H}_2$ that are invariant under $T$ and $V$. Also, $\Sigma|_{\mathcal{H}_j}$ is a pure $\Gamma$-isometry that $q_j$ annihilates for $j=1,2$. Continuing this process for finitely many steps, we have the desired conclusion.
	\end{proof}

\vspace{0.3cm}

\noindent \textbf{Funding.} The first named author is supported in part by Core Research Grant with Award No. CRG/2023/005223 from Anusandhan National Research Foundation (ANRF) of Govt. of India. The second named author is supported by the Prime Minister's Research Fellowship (PMRF) with PMRF Id No. 1300140 of Govt. of India. 

\medskip

\noindent \textbf{Declarations.} No data was analysed or used during the course of the paper. All authors of this
 paper contributed equally to article. Also, there is no competing interest.

	\end{document}